\documentclass[11pt,reqno]{amsart}
\usepackage[utf8]{inputenc}
\usepackage{float}
\usepackage{amsmath,hyperref,enumitem}
\usepackage{graphicx}
\usepackage{latexsym}
\usepackage{amsfonts}
\usepackage{amssymb}
\usepackage{color}
\usepackage[capitalise,noabbrev]{cleveref}
\newcommand{\version}{Version of \today}
\theoremstyle{plain}
\newtheorem{theorem}{Theorem}
\newtheorem{corollary}[theorem]{Corollary}
\newtheorem{lemma}[theorem]{Lemma}
\newtheorem{proposition}[theorem]{Proposition}
\theoremstyle{definition}

\newtheorem{remark}[theorem]{Remark}
\newtheorem*{remark*}{Remark}

\usepackage{multicol,ulem,tikz}
\usepackage{tikz,wasysym}
\usetikzlibrary{mindmap}
\usetikzlibrary{shapes,backgrounds}
\usepackage{tikz-3dplot}
\usepackage{tabulary}
\usetikzlibrary{shapes,calc,positioning}
\usetikzlibrary{calc,intersections}
\usepackage{tikz}
\usetikzlibrary{3d,calc,fadings,decorations.pathreplacing}

\tikzset{%
  >=latex, 
  inner sep=0pt,%
  outer sep=2pt,%
  mark coordinate/.style={inner sep=0pt,outer sep=0pt,minimum size=3pt,
    fill=black,circle}%
}

\usepackage{amssymb,amsfonts,amsmath}
\renewcommand {\epsilon}{\varepsilon}

\renewcommand {\leq}{\leqslant}
\renewcommand {\geq}{\geqslant}

\newcommand{\pr}{\mathbf P}
\newcommand{\e}{\mathbf E}

\DeclareMathOperator{\dist}{dist}

\DeclareMathOperator{\Vol}{Vol}
\DeclareMathOperator{\bm}{bm}

\begin{document}
\title[Boundary behavior of random walks in cones]
{Boundary behavior of random walks in cones}
\thanks{This project has received funding from the European Research Council (ERC) under the European Union's Horizon 2020 research and innovation programme under the Grant Agreement No 759702.}
\thanks{\version}

\author[K. Raschel]{Kilian Raschel} 
\address{CNRS, Institut Denis Poisson, Universit\'e de Tours et Universit\'e d'Orl\'eans, France}
\email{raschel@math.cnrs.fr}
       
\author[P. Tarrago]{Pierre Tarrago} 
\address{Laboratoire de Probabilit\'es, Statistique et Mod\'elisation, Sorbonne Universit\'e,~France}
\email{pierre.tarrago@upmc.fr}

\begin{abstract}
We study the asymptotic behavior of zero-drift random walks confined to multidimensional convex cones, when the endpoint is close to the boundary. We derive a local limit theorem in the fluctuation regime.
\end{abstract}

\keywords{Random walk, cone, exit time, harmonic function, Brownian motion, coupling, heat kernel}
\subjclass{Primary 60G50; secondary 60G40, 60F17} 

\maketitle

\section{Introduction}

\subsubsection*{Local limit theorems for random walks}
Given a random walk $\{S(n)\}_{n\geq0}$ on ${\bf Z}^d$ started at the origin ($S(0)=0$), a local limit theorem consists in the asymptotic derivation of the local probabilities
\begin{equation}
\label{eq:local_proba}
    \pr(S(n) = y)
\end{equation}
as $n$ goes to infinity. The classical situation corresponds to a fixed ending point $y\in{\bf Z}^d$, but other interesting regimes exist when $y=y_n$ is allowed to depend on $n$ (for example, as $\vert y\vert$ goes to infinity with $n$). For instance, if the random walk is adapted and aperiodic, has finite second moment and zero drift (i.e., $\e[ S(1)^2]<\infty$ and $\e[S(1)]=0$), then the following \textit{local convergence result} holds
\begin{equation}
\label{eq:LLT_not_constrained}
     \pr(S(n) = y) \sim \frac{1}{(2\pi n)^{d/2}} \frac{1}{\vert \det Q\vert^{1/2}} \exp\left(-\frac{\langle y,Q^{-1}y\rangle}{2n}\right),
\end{equation}
where $Q$ is the covariance matrix, see \cite{GnKo-54,Sp-76}. Two features of the above asymptotics should be noted: first, the critical exponent (i.e., the exponent of $1/n$) is $d/2$; second, the term $\exp\bigl(-\frac{\vert y\vert^2}{2n}\bigr)$ (modified by the inverse of the covariance matrix) is typical for such local convergences with \textit{Gaussian estimates}. In particular, $\pr(S(n) = y)$ goes to zero polynomially fast in the fluctuation regime (when ${\vert y\vert}/{\sqrt{n}}$ is bounded) and exponentially fast in the large deviation regime (as $\vert y\vert/\sqrt{n}$ goes to infinity). Incidentally, the previous result \eqref{eq:LLT_not_constrained} implies the recurrence of zero-mean random walks in dimensions one and two, and transience if $d\geq 3$, a celebrated result due to P\'olya \cite{Po-21} for simple random walks. 
The local limit theorem \eqref{eq:LLT_not_constrained} can be adapted to random walks with non-zero drift \cite{NeSp-66}, periodic lattice random walks \cite{Sp-76}, non-lattice random walks \cite{St-65}, time- and space-inhomogeneous random walks \cite{Mu-06}, etc.; we will not consider these issues here.


Let now $K$ be a subdomain of ${\bf Z}^d$ and let $\tau_x$ be the associated exit time from $K$ for the random walk started at $x$, i.e., 
\begin{equation}
\label{eq:exit_time_def}
     \tau_x=\inf\{n\geq 1 : x+S(n)\notin K\}.
\end{equation} 
In comparison with the unconstrained case, establishing a \textit{constrained} local limit theorem, i.e., studying the asymptotic behavior of
\begin{equation}
\label{eq:local_proba_constrained}
     \pr(x+S(n) = y,\tau_x>n),
\end{equation}
happens to be much more complex. The presence of $x$ in the probability \eqref{eq:local_proba_constrained} is due to the fact that contrary to ${\bf Z}^d$, a proper subdomain $K$ is not translation-invariant, and so the starting point will have an influence.

In the present work, we will deal with conical subdomains $K$. Random walks in cones are indeed very important in probability theory, as they appear in multiple natural contexts: nonintersecting paths or random walks in Weyl chambers \cite{St-90,EiKo-08,KoSc-10,DeWa-10,Fe-14}, non-colliding random processes \cite{KoOCRo-02}, eigenvalues of Dyson Brownian motion \cite{Dy-62}, random walks in the quarter plane \cite{BMMi-10,FaIaMa-17}, queueing theory \cite{CoBo-83}, finance \cite{CodL-13}, modeling of some populations in biology \cite{BiTr-12}, etc. As these random walk models are in bijection with many other discrete models (maps, permutations, trees, Young tableaux, partitions), they are also intensively studied in combinatorics \cite{BMMi-10,BoBMKaMe-16,DrHaRoSi-18}.

Let us now review the literature regarding local limit theorems in cones. In dimension one, there is essentially a unique cone, namely, the positive half-line. As a consequence, one-dimensional random walks in cones are equivalent to the model of \textit{positive random walks}, on which there exists a large literature, see \cite{Ig-74,AlDo-99,BrDo-06,VaWa-09,Do-12} and references therein. The theory of fluctuations of positive random walks is actually well understood; in dimension one, the \textit{Wiener-Hopf factorization} turns out to represent a crucial tool, allowing to deduce the asymptotics of \eqref{eq:local_proba_constrained} in various regimes.

The problem is of much greater complexity in dimension $d\geq2$, one of the reasons being the lack of a multidimensional version of the Wiener-Hopf factorization. However, let us mention a few major contributions in that field. In \cite{Va-99}, motivated by applications to the analysis on Lie groups, Varopoulos derives an upper bound for the probability~\eqref{eq:local_proba_constrained}. More precisely, \cite[Thm~5]{Va-99} states the following Gaussian estimate for random walks with zero drift, identity covariance matrix and bounded increments:
\begin{equation}
\label{eq:lower_bound_Varo}
     \pr(x+S(n) = y,\tau_x>n)\leq C\frac{u(x+x_0) u(y+x_0)}{n^{p+d/2}} \exp\left( -\frac{\vert x-y\vert^2}{cn}\right),
\end{equation}
where $c$ and $C$ are positive constants, $x_0$ is some fixed point in the interior of the cone and $u$ is the unique harmonic function positive in the cone $K$ and vanishing on the boundary; $u$ is usually called the \textit{r\'eduite} of the cone. The quantity $p$ in \eqref{eq:lower_bound_Varo} is positive and may be interpreted as the homogeneity exponent of the function $u$, i.e., for all $x\in K$ and $t>0$, 
\begin{equation}
\label{eq:homogeneity_reduite}
     u(tx)=t^pu(x).
\end{equation}
The exponent $p$ depends continuously on the geometry of the cone (indeed, it is directly related to the smallest eigenvalue of an eigenvalue problem with Dirichlet conditions, see \cite{De-87}). For instance, in dimension two, $p$ equals $\pi$ divided by the opening of the wedge. This has the interesting consequence that the critical exponent $p+d/2$ in \eqref{eq:lower_bound_Varo} may be an irrational number. To the authors' knowledge, Varopoulos' paper provides one of the first appearances of irrational critical exponents in the random walk literature. Assuming the covariance matrix to be the identity in \eqref{eq:lower_bound_Varo} is not really a constraint, as doing a linear change of variable allows to pass from an arbitrary covariance matrix to the identity (obviously, this change of variable will have an effect on the cone).

Ben-Salem, Mustapha and Sifi \cite{BSMuSi-14} obtain upper Gaussian estimates of transition probabilities of some models of \textit{space-inhomogeneous random walks} on the positive quadrant. These bounds have a form comparable to \eqref{eq:lower_bound_Varo}. Among the most important steps in their proof are comparison arguments based on discrete variants of the Harnack principle and large deviations estimates. See also the combinatorial results of Mustapha \cite{DALaMu-16,Mu-19}.

\textit{Weyl chambers} form another class of examples for which the multidimensional local limit theorem \eqref{eq:local_proba_constrained} is known, see \cite{St-90,EiKo-08,KoSc-10,DeWa-10,Fe-14}. The particular structure and the rigidity of these cones allow to do precise computations; in the simplest cases of reflectable walks, one can even apply the reflection principle \cite{GrMa-93,Fe-14}. See \cite[Sec.~1.6]{DeWa-15} for a more complete exposition on random walks in Weyl chambers.

Another source of examples is given by combinatorics. Indeed, there has been recently an important interest in the combinatorial model of walks in the quarter plane, and more generally \textit{walks confined in orthants}, see \cite{BMMi-10,BoBMKaMe-16,DALaMu-16,DrHaRoSi-18,Mu-19}. In particular, for some models, exact expressions for the generating functions counting these numbers of walks are obtained \cite{BMMi-10,BoBMKaMe-16}. Up to a scaling, these results are easily turned into probabilistic local limit theorems. However, such exact expressions require a very strong structure of the associated model (in some sense, an algebraic variant of the reflection principle should apply), and as a consequence they are derived for a relatively small number of models.

Let us now mention the breakthrough paper \cite{DeWa-15}, in which the authors prove for a large class of cones and random walks with zero drift, identity covariance matrix and $\max\{2+\varepsilon,p\}$ moments, that
\begin{equation}
\label{eq:LLT_DeWa}
     \pr(x+S(n) = y,\tau_x>n)\sim \kappa\cdot V(x)\cdot V'(y)\cdot n^{-p-d/2},
\end{equation}
where $\kappa$ is a positive constant, $V$ and $V'$ are \textit{discrete harmonic functions} (we shall study these functions in more details in Section \ref{sec:statements}), and $p$ is as in \eqref{eq:homogeneity_reduite}, see \cite[Thm~6]{DeWa-15}. Obviously the critical exponents in \eqref{eq:lower_bound_Varo} and \eqref{eq:LLT_DeWa} coincide. A version of the above result exists for random walks with drift and non-identity covariance matrix, performing an exponential change of measure and a linear change of coordinates. The methods employed in proving \eqref{eq:LLT_DeWa} use various arguments and fine estimates concerning the behavior of random walks. Of particular interest is a \textit{coupling of the random walk with Brownian motion}, for which the authors use results from \cite{GoZa-09} on the quality of the normal approximation. Another important point is to construct (from the classical harmonic function of Brownian motion) and to study the discrete harmonic functions $V$ and $V'$. Notice that $V$ and $V'$ in \eqref{eq:LLT_DeWa} are asymptotically equivalent to the r\'eduite $u$ in \eqref{eq:homogeneity_reduite}, by \cite[Lem.~13]{DeWa-15}.

\subsubsection*{A glimpse of our results}

In this article, we refine the local limit theorem \eqref{eq:LLT_DeWa} of Denisov and Wachtel, by allowing $y$ to vary with $n$. More precisely, we will derive the asymptotic behavior of \eqref{eq:local_proba_constrained} in the fluctuation regime, meaning that $\vert y\vert/\sqrt{n}$ is bounded. Under the additional assumption that $\dist(y,\partial {K})\geq n^{1/2-\epsilon}$ for some small $\epsilon>0$, we shall prove a uniform Gaussian local convergence
\begin{equation}
\label{eq:Gaussian_estimate_Cor2}
     \pr(x+S(n)=y,\tau_{x}>n)\sim\kappa\cdot V(x)\cdot u(y)\cdot n^{-p-d/2}\cdot \exp\left(-\frac{\vert y\vert^{2}}{2n}\right),
\end{equation}
see Corollary \ref{cor:asymp_probab_boundary}. We have a more precise version of \eqref{eq:Gaussian_estimate_Cor2}, in which $y$ may tend to infinity but does not need to be far away from the boundary, see Theorem \ref{thm:asymp_probab_boundary}. Theorem \ref{thm:asymp_probab_boundary} is the central result of our article (it contains the initial statement \eqref{eq:LLT_DeWa} of Denisov and Wachtel) and is at the origin of its title. Indeed, the influence of the geometry of the cone is much stronger in the boundary case of Theorem \ref{thm:asymp_probab_boundary}, in comparison with the interior case \eqref{eq:Gaussian_estimate_Cor2}. We don't state Theorem \ref{thm:asymp_probab_boundary} right here, because it needs a number of additional notations. Our method consists in \textit{combining Gaussian estimates and the coupling} of the random walk by Brownian motion.

\subsubsection*{Applications of our results}

Let us give four main features of our results. Our first motivation is related to the \textit{Green function} 
\begin{equation*}
     G(x,y) = \sum_{n=0}^\infty \pr(x+S(n)=y,\tau_x>n)
\end{equation*}
of the random walk in the cone $K$. The precise asymptotics of the local probabilities \eqref{eq:local_proba_constrained} that we derive in this paper is a first step on the way of determining the asymptotic behavior of the Green function as $\vert y\vert$ grows to infinity. This question is the topic of the separate work \cite{DuRaTaWa-20}, in collaboration with Jetlir Duraj and Vitali Wachtel; moreover, from the Green function's asymptotics, we also deduce in \cite{DuRaTaWa-20} the uniqueness of the discrete harmonic function for random walks in cones, and thereby the structure of the Martin boundary. 

Secondly, our Theorem \ref{thm:asymp_probab_boundary} \textit{generalizes the local limit theorem} \eqref{eq:LLT_DeWa}. This understanding at a higher level allows us to shed light into the underlying mechanisms of Denisov and Wachtel's local limit theorem.

Our third application concerns the enumeration of walks in cones. Given a cone (typically the orthant ${\bf N}^d$) and a set of steps, the question is to compute (in an exact or asymptotic way) the number of excursions, i.e., the number of paths of length $n$ going from $x$ to $y$ and confined to the cone. All our results concerning the local probability $\pr(x+S(n)=y,\tau_{x}>n)$ (especially Theorem \ref{thm:asymp_probab_boundary} and Corollary \ref{cor:asymp_probab_boundary}) can straightforwardly be turned into combinatorial results, giving the asymptotic behavior of the number of excursions. This discussion (in particular, the precise combinatorial traduction of the probabilistic results) is detailed in \cite[Sec.~1.5]{DeWa-15}, so we refer to \cite{DeWa-15} for further information.

Let us finally notice that the latter combinatorial interpretation can lead to new \textit{asymptotics in combinatorial representation theory}. A recurrent problem is to give asymptotic formula for the fusion rules of a Lie algebra, which describe the decomposition of an iterated tensor product of a given representation into irreducible representations. In particular, a central question is to give the Martin boundary of the multiplicative graph generated by the latter construction (see \cite{LeTa-16} for a precise definition of the multiplicative graph and a description of the minimal boundary). In his seminal paper \cite{Lit-94}, Littelmann shows that the decomposition of irreducible representations is encoded by the concatenation of paths in the Weyl chamber of the Lie algebra, which is a convex cone. Therefore, the local limit theorems given in the present paper should give the asymptotic results in the fluctuation regime. Although these results are not enough to deduce the Martin boundary of the corresponding multiplicative graph, they provide a first step in this direction. An important improvement would be to get the local limit theorems in the large deviation regime, which would indeed complete the description of the Martin boundary. 

\subsubsection*{Structure and sketch of the results}
Our paper is organized as follows: in Section \ref{sec:statements}, we state our main results (Theorem \ref{thm:asymp_probab_boundary} and Corollary \ref{cor:asymp_probab_boundary}). We prove in Section \ref{sec:preliminary_results} Gaussian estimates for the heat kernel in a cone and we report on the coupling approach of \cite{DeWa-15}. In Section \ref{sec:mesoscopic}, we prove our main results. Some technical proofs are postponed to Appendix \ref{appendix_heat_estimates}. In Appendix \ref{Fuk-Nagaev-cone}, we give some useful Fuk-Nagaev inequalities for random walks in a cone.

\subsubsection*{Acknowledgments} 
We would like to thank Sami Mustapha and Vitali Wachtel for useful discussions and bibliographic advices. We also express our deepest gratitude to the referee for very careful readings and many useful suggestions.

\section{Statement of the main results}
\label{sec:statements}

\subsubsection*{Notations and assumptions on the cones and random walks}
Let us start by presenting our hypotheses, which are of three types: some of them only concern the random walk (namely \ref{H:lattice}, \ref{H:drift_zero}, \ref{H:cov_id} and \ref{H:RW_aperiodic}), the assumption \ref{H:regularity_cone} is a restriction on the cone, while the last ones (\ref{H:strongly_irreducible} and \ref{H:moments}) concern the behavior of the random walk in the cone.

Consider a random walk $\{S(n)\}_{n\geq1}$ on ${\bf R}^d$, $d\geq1$, where
\begin{equation*}
     S(n) = X(1)+\cdots +X(n)
\end{equation*}
and $\{X(n)\}_{n\geq1}$ is a family of independent and identically distributed (i.i.d) copies of a random variable $X=(X_1,\ldots,X_d)$. We assume that: 
\begin{enumerate}[label=($H\arabic{*}$),ref=($H\arabic{*}$)]
     \item\label{H:lattice}the random variable $X$ is lattice,
     \item\label{H:drift_zero}$\e[X]=0$,
     \item\label{H:cov_id}$\text{cov}(X_{i},X_{j})=\delta_{i,j}$,
     \item\label{H:RW_aperiodic}the random walk is aperiodic.
\end{enumerate}     
Notice that \ref{H:cov_id} is not a restriction: we may always perform a linear transform so as to decorrelate the random walk (obviously this linear transform changes the cone in which the walk is defined). 

Denote by ${\bf S}^{d-1}$ the unit sphere of ${\bf R}^d$ and by $\Sigma$ an open, connected subset of ${\bf S}^{d-1}$. Let $K$ be the cone generated by the rays emanating from the origin and passing through $\Sigma$, i.e., $\Sigma=K\cap {\bf S}^{d-1}$; see Figure \ref{fig:cones} for two examples. In this paper, we suppose that:
\begin{enumerate}[label=($H\arabic{*}$),ref=($H\arabic{*}$)]
\setcounter{enumi}{4}
     \item\label{H:regularity_cone}the cone $K$ is convex.
     \end{enumerate} 
When $K$ is convex, at each point $q\in\partial \Sigma$, there exists a non-trivial closed ball $B$ in ${\bf S}^{d-1}$ such that $B\cap \Sigma=q$. Hence, by standard analytic results \cite[Thm~6.13]{GiT-01}, $\Sigma$ is regular for the Dirichlet problem. In particular, there exists a function $u$ harmonic on ${K}$, i.e., $\Delta u=0$, such that $u$ is positive in $K$ and $u_{\partial {K}}=0$, $\partial {K}$ denoting the boundary of $K$ (see for example the introduction of \cite{BaSm-97}). This function is unique up to scalar multiplication, see \cite[Cor.~6.10 and Rem.~6.11]{GySal-11}, and is called the r\'eduite of ${K}$. It is homogeneous (or radial) in the sense of \eqref{eq:homogeneity_reduite} and the homogeneity exponent $p$ is also called the exponent of the cone ${K}$. 

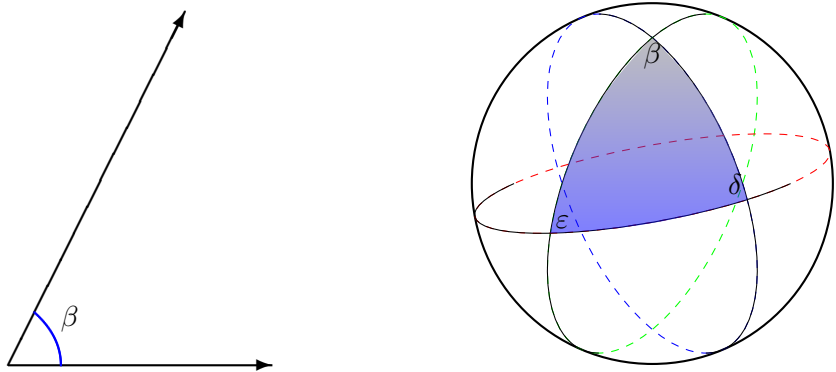
\begin{figure}[ht]
\begin{picture}(0,0)
    \thicklines
    \put(0,0){\textcolor{black}{\vector(1,0){100}}}
    \put(0,0){\textcolor{black}{\vector(1,2){67}}}
    \thicklines
    \put(20,15){$\beta$}
    \textcolor{blue}{\qbezier(20,0)(20,11)(10,20)}
    \end{picture} \qquad\qquad\qquad\qquad\qquad\qquad\qquad
\tdplotsetmaincoords{90}{90}
\begin{tikzpicture}[scale=3,tdplot_main_coords]

\tdplotsetthetaplanecoords{90}
\tdplotdrawarc[tdplot_rotated_coords,thick]{(0,0,0)}{0.8}{0}{360}{}{}
\tdplotsetrotatedcoords{60}{70}{0}
\tdplotdrawarc[dashed,tdplot_rotated_coords,name path=blue,color=blue]{(0,0,0)}{0.8}{0}{360}{}{}
\tdplotdrawarc[tdplot_rotated_coords]{(0,0,0)}{0.8}{0}{180}{}{}
\tdplotsetrotatedcoords{120}{110}{0}
\tdplotdrawarc[dashed,tdplot_rotated_coords,name path=green,color=green]{(0,0,0)}{0.8}{0}{360}{}{}
\tdplotdrawarc[tdplot_rotated_coords]{(0,0,0)}{0.8}{0}{180}{}{}
\tdplotsetrotatedcoords{220}{16}{0}
\tdplotdrawarc[dashed,tdplot_rotated_coords,name path=red,color=red]{(0,0,0)}{0.8}{0}{360}{}{}
\tdplotdrawarc[tdplot_rotated_coords]{(0,0,0)}{0.8}{-90}{90}{}{}


\path [name intersections={of={green and blue}, total=\n}]  
\foreach \i in {1,...,\n}{(intersection-\i) circle [radius=0.5pt] coordinate(gb\i){}};

\path [name intersections={of={green and red}, total=\n}]  
\foreach \i in {1,...,\n} {(intersection-\i) circle [radius=0.5pt]coordinate(gr\i){}};

\path[name intersections={of={red and blue}, total=\n}]  
\foreach \i in {1,...,\n}{(intersection-\i) circle [radius=0.5pt]coordinate(rb\i){}};

\shade[top color=gray,bottom color=blue,opacity=0.5]  
(rb3) to[bend left=7] (gr1) to[bend left=17] (gb2) to[bend left=15] cycle;

\draw (gb2) node[below]{$\beta$};
\draw (rb3) node[above left]{$\delta$};
\draw (gr1) node[above right]{$\varepsilon$};
\end{tikzpicture}
\caption{In dimension $2$, $\Sigma$ is an arc of circle and the cone $K$ is a wedge of opening $\beta$. In dimension $3$, any section $\Sigma\subset{\bf S}^2$ can be taken. The picture on the right gives the example of a spherical triangle on the sphere ${\bf S}^2$, corresponding to the orthant $K={\bf N}^3$ (after possible decorrelation of the coordinates, see \ref{H:cov_id}).}
\label{fig:cones}
\end{figure}

Our final hypotheses deal with the behavior of the random walk in the cone $K$. First, we require a form of irreducibility for the random walk, which is an adaptation to unbounded random walks of the concept of reachability condition from infinity introduced in \cite{BoBMMe-18}. From now on, we fix an origin $x_{0}\in{\bf R}^{d}$ and denote by $\Lambda$ the lattice generated by the random walk starting at $x_{0}$. We denote by $\{S'(n)\}_{n\geq 1}$ the reverse random walk, which is the sum of the increments $\{X'(n)\}_{n\geq 1}$, i.i.d, independent from $\{X(n)\}_{n\geq 1}$ and such that $X'(n)$ is distributed as $-X$. (In the sequel, every quantity involving $S'$ will be denoted similarly as the same quantity involving $S$, with a prime added at the right.)
\begin{enumerate}[label=($H\arabic{*}$),ref=($H\arabic{*}$)]
\setcounter{enumi}{5}
     \item\label{H:strongly_irreducible}The reversed random walk $S'$ is asymptotically strongly irreducible, meaning that there exists a constant $R>0$ such that for any $z\in K\cap \Lambda$, $\vert z\vert\geq R$, there exists a path with positive probability in $K\cap B(z,R)$ which starts in $z$ and ends at $z+K$.
\end{enumerate}
 There are several simple situations where the latter condition is satisfied. In particular, this is the case when $\pr(X\in -K)>0$. If $K$ is $\mathcal{C}^{2}$, the condition \ref{H:strongly_irreducible} is superfluous. We shall also assume a moment condition on the increments which is slightly stronger than \cite{DeWa-15}:
\begin{enumerate}[label=($H\arabic{*}$),ref=($H\arabic{*}$)]
\setcounter{enumi}{6}
     \item\label{H:moments}$\e[\vert X\vert^{p+2}]<\infty$.   
\end{enumerate} 
This moment condition can be explained by the current limitation on the estimates for the moment of the exit time for the random walk starting away from the origin.

However, we do not require the existence of a bigger cone $K'$ with $\partial K\setminus\lbrace 0\rbrace\subset \text{int}(K')$ and such that the r\'eduite $u$ can be extended to a harmonic function on $K'$. This condition, which was necessary in \cite{DeWa-15}, is removed in \cite{DeWa-19}.

\subsubsection*{Harmonic functions and reverse random walk}

A  function $h:{K}\to {\bf R}$ is said to be (discrete) harmonic with respect to ${K}$ and $\{S(n)\}$ if for every $x\in{K}$ and $n\geq 1$,
\begin{equation*}
     h(x)=\e[h(x+S(n));\tau_{x}>n],
\end{equation*}
with $\tau_x$ defined in \eqref{eq:exit_time_def}.
Remark that the above identity for $n=1$ implies all the other relations for $n\geq 2$. In the sequel, a harmonic function with respect to ${K}$ and $\{S(n)\}$ will be simply called a harmonic function.

Denisov and Wachtel proved \cite[Thm~1]{DeWa-15} the existence of a positive harmonic function $V:{K}\to [0,\infty)$ defined by 
\begin{equation}\label{definition_V(x)}
     V(x)=\lim\limits_{n\rightarrow \infty}\e[u(x+S(n));\tau_{x}>n].
\end{equation}
This harmonic function is of central importance in the present paper, and it will ultimately be identified in \cite{DuRaTaWa-20} with the Martin boundary of the random walk in ${K}$.

\subsubsection*{Local limit theorem at the fluctuation scale}
The main result of the paper is a uniform extension of the local limit theorem \cite[Thm~6]{DeWa-15}. This extension can be made more or less explicit depending on the proximity with the boundary of the cone with respect to the fluctuation scale. Let $\epsilon>0$ be a small parameter. Following \cite{DeWa-15}, we set 
\begin{equation}
\label{eq:def_K-n-epsilon}
     {K}_{n,\epsilon}=\lbrace x\in{K}: \dist(x,\partial {K})\geq n^{1/2-\epsilon}\rbrace.
\end{equation}
Since the typical fluctuations of the random walk at time $n$ are expected to be of order $\sqrt{n}$, ${K}_{n,\epsilon}$ represents the set of points of the cone whose distance to the boundary at time $n$ is more than $n^{-\epsilon}$ times the typical fluctuations. 

For $x\in{K}$ and $n\geq 1$, set
\begin{equation}
\label{eq:def_t_x_n}
     t_{x,\epsilon}(n)=\inf \lbrace m\geq 0: x+S(m)\in{K}_{n,\epsilon}\rbrace.
\end{equation}
Then $t_{x,\epsilon}(n)$ is a stopping time with respect to $\{S(m)\}$, which gives the first time at which the random walk started at $x$ is far from the boundary compared to the order of fluctuation $\sqrt{n}$; we denote then by $x_{\epsilon}(n)$ the position $x+S(t_{x,\epsilon}(n))$. Define similarly $t'_{x,\epsilon}(n)$ and $x'_{\epsilon}(n)$ for the reverse random walk.
\begin{theorem}
\label{thm:asymp_probab_boundary}
Assume \ref{H:lattice}--\ref{H:strongly_irreducible} and assume moreover that $\e[\vert X\vert^{r}]<\infty$ for some $r>p+2$. Then there exists a constant $\kappa$ which depends only on $K$ such that for any fixed $x\in{K}$, as $n\to\infty$,
\begin{multline*}
     \pr(x+S(n)=y,\tau_{x}>n)\sim\\\kappa\cdot V(x)\cdot n^{-p/2-d/2}\cdot \e[u(y'_{\epsilon}(n)/\sqrt{n});t'_{y,\epsilon}(n)\leq \tau'_{y}]\cdot \exp\left(-\frac{\vert y\vert^{2}}{2n}\right),
\end{multline*}
uniformly on $y\in {K}$ such that $\vert y\vert \leq A\sqrt{n}$ with $A>0$.
\end{theorem}

Theorem \ref{thm:asymp_probab_boundary} implies the local limit theorem of \cite{DeWa-15} under the stronger assumption that $\e[\vert X\vert^{r}]<\infty$ for some $r>p+2$. Indeed, for a fixed value of $y$, \cite[Thm~1]{DeWa-15} yields
\begin{equation*}
     n^{p/2}\cdot\e[u(y'_{\epsilon}(n)/\sqrt{n});t'_{y,\epsilon}(n)\leq \tau'_{y}]\cdot\exp\left(-\frac{\vert y\vert^{2}}{2n}\right)\xrightarrow[n\rightarrow \infty]{} V'(y).
\end{equation*}
Theorem \ref{thm:asymp_probab_boundary} admits a nice simplification when the endpoint of the random walk is located in the domain ${K}_{n,\epsilon}$ in \eqref{eq:def_K-n-epsilon}.

\begin{corollary}\label{cor:asymp_probab_boundary}
Assume \ref{H:lattice}--\ref{H:moments}. Then for any fixed $x\in{K}$, as $n\to\infty$,
\begin{equation*}
     \pr(x+S(n)=y,\tau_{x}>n)\sim\kappa _{0}\cdot V(x)\cdot n^{-p-d/2}\cdot u(y)\cdot \exp\left(-\frac{\vert y\vert^{2}}{2n}\right),
\end{equation*}
uniformly on $y\in {K}_{n,\epsilon}$ with $\vert y\vert \leq A\sqrt{n}$, $A>0$.
\end{corollary}

We shall prove Theorem \ref{thm:asymp_probab_boundary} by following the same pattern as that of \cite[Thm~6]{DeWa-15}: all intermediate steps of our proof are improved versions of the corresponding intermediate steps in the proof of \cite[Thm~6]{DeWa-15}. In \cite{DeWa-15}, these steps mainly consist in a coupling of the random walk with a standard Brownian motion, and then of estimates of the survival probability of the random walk started at a given $x\in{K}$, or of the transition probability between two fixed points $x,y\in{K}$, when $n$ goes to infinity. The main novelty in the proof of Theorem \ref{thm:asymp_probab_boundary} is to allow $x$ and $y$ to change with $n$. To that purpose, we use Gaussian estimates on the heat kernel on ${K}$. We will mainly use results from \cite[Sec.~5 and Sec.~6.3]{GySal-11}. See also \cite{Va-99,BSMuSi-14} for complementary approaches on the subject.

\subsubsection*{Further directions}
Our results may be improved in two directions. First, as we mentioned below \ref{H:moments}, we should expect an optimal moment condition of the form $\e[\vert X\vert^{p}]<\infty$ (with $p$ instead of $p+2$). Our stronger assumption actually comes from the estimate \eqref{exit_time}, which does not take into account the distance to the boundary of the cone. In order to improve this, one would need a better understanding of the behavior of the exit time away from the origin. 

A better understanding of the function $\e[u(y'_{\epsilon}(n)/\sqrt{n});t'_{y,\epsilon}(n)\leq \tau'_{y}]$ appearing in Theorem \ref{thm:asymp_probab_boundary} would be the second improvement. For $y$ fixed and $n$ going to infinity, the latter function converges to $V'(y)$. An important problem is to get a uniform convergence to the harmonic function $V'$; such a uniform convergence exists when $K$ is the half-space, as it is shown in \cite{DuRaTaWa-20}.

\section{Preliminary results}
\label{sec:preliminary_results}

In this section, we give a brief review on Gaussian estimates for the heat kernel in a cone and on the coupling approach of \cite{DeWa-15}.

\subsection{Heat kernel in a cone and Gaussian estimates.}\label{gaussian_estimates}
For $t>0$, let us denote by $K_{t}:{K}\times {K}\to [0,\infty)$ the heat kernel on ${K}$ at time $t$, with Dirichlet boundary conditions. Namely, $K_{t}$ is the solution of the equation
\begin{equation*}
     \left\lbrace \begin{array}{ll}
\partial_{t}K_{t}(x,\cdot)+\Delta K_{t}(x,\cdot)=0,& x\in{K},\smallskip\\
\lim\limits_{t\rightarrow 0}K_{t}(x,\cdot)=\delta_{x},& x\in{K},\\
K_{t}(x,y)=0,& x\in{K},\,  y\in\partial {K},
\end{array}\right.
\end{equation*}
where the limit on the second line is understood in the distributional sense. The kernel $K_{t}$ is symmetric in $x$ and $y$.

We denote by $k_{t}:{K}\to [0,1]$ the survival probability of a standard Brownian motion at time $t$. If $\tau^{\bm}_{x}$ denotes the survival time of a standard Brownian motion starting at $x$, i.e., $\tau_{x}^{\bm}=\inf\{t\geq 0: x+B_{t}\not\in K\}$, then 
\begin{equation*}
     k_{t}(x)=\pr(\tau_{x}^{\bm}\geq t). 
\end{equation*}
It satisfies the formula
\begin{equation*}
     k_{t}(x)=\int_{{K}}K_{t}(x,y)dy,\end{equation*} 
see \cite[Eq.~(5.7)]{GySal-11}. The functions $K_{t}$ and $k_{t}$ are homogeneous in time, in the sense that 
\begin{equation*}
     k_{t}(x)=k_{1}(x/\sqrt{t})\quad\text{and}\quad K_{t}(x,y)=t^{-d/2}K_{1}(x/\sqrt{t},y/\sqrt{t})
\end{equation*}
for all $x,y\in {K}$ and $t>0$. We shall write $\bar{K}$ for $K_{1}$ and $\bar{k}$ for $k_{1}$.
 
Let $u$ denote the r\'eduite of the convex cone ${K}$, whose formula is given by
\begin{equation*}u(x)=\vert x\vert^{p}m\left(\frac{x}{\vert x\vert}\right),\end{equation*}
with $p\geq 1$ and $m: \Sigma\to {\bf R}$ is a function which is $\mathcal C^{2}$ in the interior of $\Sigma$ and such that $m_{\partial \Sigma}=0$. Note that the gradient of $u$ is locally bounded (see Lemma \ref{bound_gradient_u}). 

The first important Gaussian estimates give upper and lower bounds of $k_{t}$ in terms of the r\'eduite $u$. For $x\in{K}$ and $t>0$, let $x_{t}$ denote an arbitrary point of ${K}$ such that $\vert x_{t}-x\vert\leq t$ and $\dist(x_{t},\partial {K})\geq c_{0}t$, for some constant $c_{0}$ independent of $x$ and $t$ (we take the same notation as in \cite[Eq.~(4.29)]{GySal-11}). In all results involving elements of type $x_{t}$, the constants given will not depend on the particular choice of $x_{t}$. Moreover, since $\dist(x_{t},\partial {K})\geq c_{0}t$ by definition, \cite[Lem.~19]{DeWa-15} yields that from some constant $c>0$,
\begin{equation}
\label{lower_bound_y_1}
     u(x_{t})\geq ct^{p}.
\end{equation}

\begin{theorem}[Thm~5.14 of \cite{GySal-11}]
\label{estimatek}
There exist positive constants $c_{1}$ and $C_{1}$ such that for all $x\in{K}$ and $t>0$,
\begin{equation*}c_{1}\frac{u(x)}{u(x_{\sqrt{t}})}\leq k_{t}(x)\leq C_{1}\frac{u(x)}{u(x_{\sqrt{t}})}.\end{equation*}
\end{theorem}
We will also use an equivalent version of the latter theorem, which relies on the properties of the harmonic function $u$. Indeed, by \cite[4.20]{GySal-11}, there exist constants $c$ and $C$ such that 
\begin{equation*}cu(x_{t})^{2}\leq\frac{\int_{B(x,t)\cap {K}}u(s)^{2}ds}{V(x,t)}\leq Cu(x_{t})^{2},\end{equation*}
where $V(x,t)$ denotes the volume of $B(x,t)\cap{K}$. This yields the alternative estimate
\begin{equation}\label{estimateWithVolume}
c_{1}'\sqrt{\frac{V(x,\sqrt{t})}{\int_{B(x,\sqrt{t})\cap {K}}u(s)^{2}ds}}u(x)\leq k_{t}(x)\leq C_{1}'\sqrt{\frac{V(x,\sqrt{t})}{\int_{B(x,\sqrt{t})\cap {K}}u(s)^{2}ds}}u(x).
\end{equation}

The second estimates concern the heat kernel itself. 
\begin{theorem}[Thm~5.11 and Thm~5.15 in \cite{GySal-11}]\label{estimateK}
There exist positive constants $c_{2},c_{3},C_{2},C_{3},C_{4}$ and $0<\alpha<1$ such that
\begin{equation*}c_{2}\frac{k_{t}(x)k_{t}(y)}{\sqrt{V(x,\sqrt{t})V(y,\sqrt{t})}}\exp\left(-\frac{\vert x-y\vert^{2}}{C_{3}t}\right)\leq K_{t}(x,y)\leq C_{2}\frac{k_{t}(x)k_{t}(y)}{\sqrt{V(x,\sqrt{t})V(y,\sqrt{t})}}\exp\left(-\frac{\vert x-y\vert^{2}}{c_{3}t}\right)
\end{equation*}
for all $x,y\in{K}$, and
\begin{equation*}\left\vert \frac{K_{t}(x,y)}{u(y)}-\frac{K_{t}(x,y')}{u(y')}\right\vert\leq C_{4}\left(\frac{\vert y-y'\vert}{\sqrt{t}}\right)^{\alpha}\frac{K_{2t}(x,y)}{u(y)},\end{equation*}
for all $y,y'\in{K}$ such that $\vert y-y'\vert \leq \sqrt{t}$.

Moreover, there exist $\beta>0$ and $C_{5}>0$ such that
\begin{equation*}\vert \partial_{t} K_{t}(x,y)\vert\leq C_{5}\frac{k_{t}(x)k_{t}(y)}{t\sqrt{V(x,\sqrt{t})V(y,\sqrt{t})}}\left(1+\frac{\vert x-y\vert^{2}}{t}\right)^{\beta+1}\exp\left(-\frac{\vert x-y\vert^{2}}{4t}\right),\end{equation*}
for all $t>0$ and $x,y\in{K}$.
\end{theorem}
The above result is actually much more general, since by \cite{GySal-11} it holds for every inner uniform domain. The above estimates can actually be simplified by the following inequality
\begin{equation}\label{lower_bound_ball_C}
     \inf_{z\in{K}}V(z,\sqrt{t})\geq ct^{d/2},
\end{equation}
for some constant $c$ independent of $t>0$ (see Lemma \ref{proof_lower_bound_ball_C} for a proof of this inequality). 

We can deduce from the above Gaussian estimates several other estimates on the heat kernel, which will be needed to prove our Theorem~\ref{thm:asymp_probab_boundary}. All results concerning these heat kernel estimates are proven in Appendix \ref{appendix_heat_estimates}.

\subsection{The coupling approach of Denisov and Wachtel}\label{description_proofs_DW}

All the results of this subsection are borrowed from \cite{DeWa-15}, and the interested reader should refer to \cite{DeWa-15} for the proofs.  The local limit theorems obtained by Denisov and Wachtel in \cite[Thm 5 and Thm 6]{DeWa-15} rely on a coupling of the random walk with Brownian motion \cite[Lem.~17]{DeWa-15}, based on an important work of G\"otze and Zaitsev \cite[Thm 4]{GoZa-09}. Namely, suppose that $\e[\vert X\vert^{2+\delta}]<\infty$ for some $0<\delta<1$. Then, one can define on the same probability space a random walk with the same distribution as $S(n)$ and a Brownian motion $B(t)$ such that, for any $\gamma$ satisfying $0<\gamma<\frac{\delta}{2(2+\delta)}$,
\begin{equation}\label{coupling_few_moments}
     \pr(\sup_{u\leq n}\vert B(u)-S(\lfloor u\rfloor)\vert \geq n^{1/2-\gamma})\leq Cn^{2\gamma+\gamma\delta-\delta/2},
\end{equation}
for some constant $C>0$. This coupling is particularly useful in the case of random walks in cones, as the distribution of Brownian motion has an explicit expression in this situation \cite{De-87,BaSm-97}. 

For example, the value of the kernel $K_{t}(x,\cdot)$ satisfies the uniform asymptotic formula \cite[Lem.~18]{DeWa-15}
\begin{equation}\label{heat_kernel_small_x}
K_{t}(x,y)\underset{\substack{\vert x\vert\leq \theta_{t}\sqrt{t},\\\vert y\vert \leq \sqrt{t/\theta_{t}}}}{\sim} \chi_{0}t^{-d/2-p}u(x)u(y)e^{-\vert y\vert^{2}/(2t)},
\end{equation}
where $\theta_{t}$ is any function of $t$ converging to zero as $t$ goes to infinity, and $\chi_{0}$ is a positive constant. Likewise, the survival probability of Brownian motion satisfies the uniform asymptotic formula
\begin{equation}\label{survival_small_x}
k_{t}(x)\underset{\vert x\vert\leq \theta_{t}\sqrt{t}}{\sim} \chi t^{-p/2}u(x),
\end{equation}
with the same $\theta_{t}$ as before and a positive constant $\chi$. 

Thanks to the coupling \eqref{coupling_few_moments}, the above asymptotic results can be transferred \cite[Lem.~20]{DeWa-15} to the random walk $\{S(n)\}$ when the random walk starts far enough from the boundary, compared to the typical scale that we are considering. For example, for $\epsilon$ small enough and with ${K}_{n,\epsilon}$ defined in \eqref{eq:def_K-n-epsilon},
\begin{equation}
\label{survival_S_small_x}
     \pr(\tau_{x}>n)\underset{\substack{x\in{K}_{n,\epsilon}\\\vert x\vert\leq \theta_{n}\sqrt{n}}}{\sim}\chi u(x)n^{-p/2}.
\end{equation}
Let us review how this technique yields exact asymptotics for $\pr(\tau_{x}>n)$, independently of the initial position of $x$.  The matter is then to find the asymptotic formula when $x$ is close to the boundary. In order to deal with these issues, Denisov and Wachtel introduce the stopping time $t_{x,\epsilon}(n)$, see \eqref{eq:def_t_x_n}. Then they prove \cite[Lem.~14]{DeWa-15} that $t_{x,\epsilon}(n)$ is small on the event that $\tau_{x}\geq n^{1-\epsilon}$. Namely, there exists a constant $C> 0$ such that 
\begin{equation}
\label{bound_stopping_time}
     \pr(t_{x,\epsilon}\geq n^{1-\epsilon},\tau_{x}\geq n^{1-\epsilon})\leq \exp(-Cn^{\epsilon}).
\end{equation}
Therefore, on an exit time of order $n$, the random walk spends most of its time in a regime which can be controlled by the coupling \eqref{coupling_few_moments}, and one can thus show that 
\begin{equation*}\pr(\tau_{x}>n)\underset{n\rightarrow \infty}{\sim}\chi\cdot n^{-p/2}\cdot \e[u(x_{n,\epsilon});t_{x,\epsilon}(n)\leq \tau_{x},t_{x,\epsilon}(n)\leq n^{1-\epsilon}],\end{equation*}
where $x_{n,\epsilon}=x+S(t_{x,\epsilon}(n))$. The last step of their method is to adapt the definition of $V(x)$ in \eqref{definition_V(x)} to the stopping time $t_{x,\epsilon}(n)$, which yields \cite[Lem.~21]{DeWa-15}
\begin{equation}\label{alternative_definition_V(x)}
\lim\limits_{n\rightarrow \infty}\e[u(x_{n,\epsilon});t_{x,\epsilon}(n)\leq \tau_{x},t_{x,\epsilon}(n)\leq n^{1-\epsilon}]=V(x),
\end{equation}
and the exact asymptotic formula given in \cite[Thm~1]{DeWa-15}:
\begin{equation}
\label{asymptotic_survival_time}
     \pr(\tau_{x}>n)\underset{n\rightarrow \infty}{\sim}\chi\cdot n^{-p/2}\cdot V(x).
\end{equation}
The local limit theorems are then obtained from the previous results; we do not review their proof here, since the pattern is roughly the same as the one of the proof of Theorem \ref{thm:asymp_probab_boundary} in the following section. However, let us stress that these proofs rely on some important estimates on the local probability of non-constrained random walks, namely, by \cite[Lem.~29]{DeWa-15}, there exist positive constants $a$ and $C$ such that for all $u\geq 0$,
\begin{equation}
\label{bound_local_probability}
     \limsup_{n\rightarrow \infty} n^{d/2}\sup_{\vert z-x\vert \geq u\sqrt{n}}\pr(x+S(n)=z)\leq C\exp(-au^{2}).
\end{equation}
In particular  \cite[Lem.~27]{DeWa-15}, using \eqref{bound_local_probability} at $u=0$ together with \eqref{asymptotic_survival_time} yields the existence of a positive constant $C(x)$ for each $x\in{K}$ such that 
\begin{equation}\label{bound_local_probability_in_cone}
\sup_{y\in{K}}\pr(x+S(n)=y,\tau_{x}\geq n)\leq C(x)n^{-p/2-d/2}.
\end{equation}
Finally, an important tool in the proofs of \cite{DeWa-15} is provided by Fuk-Nagaev inequalities for random walks in a cone. We shall also use them in the proof of Theorem \ref{thm:asymp_probab_boundary}. In Appendix \ref{Fuk-Nagaev-cone}, we state and prove a general result summarizing these inequalities.

\section{Local limit theorem at the mesoscopic scale}
\label{sec:mesoscopic}

This section is dedicated to proving Theorem \ref{thm:asymp_probab_boundary}, which gives the uniform asymptotics of the local probability $\pr(x+S(n)=y,\tau_{x}>n)$ as $n$ goes to infinity and $\vert y\vert\leq A\sqrt{n}$ (throughout the section, $A>0$ is a fixed parameter which is expected to be large). We introduce
\begin{equation*}
     {K}_{n,\epsilon}^{A}=\lbrace z\in{K}:\vert z\vert\leq A\sqrt{n},\, \dist(z,\partial{K})\geq n^{1/2-\epsilon}\rbrace.
\end{equation*}

\subsection{Uniform convergence of the exit time and the conditioned distribution far from the boundary}\label{Section:Local_limit}

\begin{proposition}\label{uniformAway}
Let $B>0$, $0<s<1$ and $D\subset{\bf R}^{d}$ be a bounded convex domain containing $0$ in its interior. Then there exists $\epsilon>0$ such that 
\begin{equation*}\pr(\tau_{y}>n)\underset{n\rightarrow\infty}{\sim} k_{n}(y)\end{equation*}
and
\begin{equation*}\pr(y+S(n)\in \sqrt{n}(x+tD),\tau_{y}>n)\underset{n\rightarrow\infty}{\sim} \int_{x+tD}\bar{K}(y/\sqrt{n},z)dz,\end{equation*}
uniformly for all $y\in{K}_{n,\epsilon}^{A}$, $x\in{K}\cap B(0,B)$ and $t\in [s,1]$.
\end{proposition}

Let us start by giving an upper bound for $u(y_{1})$, which will be useful for the proof of Proposition \ref{uniformAway} (refer to Section \ref{gaussian_estimates} for the definition of $y_{t}$, $t>0$).

\begin{lemma}\label{lowerBoundy1}
There exists a constant $C>0$ such that for all $y\in{K}$, we can choose $y_{1}$ with
\begin{equation}
\label{eq:lowerBoundy1}
     u(y_{1})\leq u(y)\vee C(\vert y\vert\vee 1)^{p-1}.
\end{equation}
\end{lemma}
\begin{proof}
Recall that $y_{1}$ is an arbitrary point such that $\dist(y,y_{1})\leq 1$ and $\dist(y_{1},\partial {K})\geq c_{0}$. Hence, when $\dist(y,\partial {K})\geq c_{0}$, we may choose $y_{1}=y$ and thus $u(y_{1})=u(y)$. Suppose now that $\dist(y,\partial {K})\leq c_{0}$. Then any choice of $y_{1}$ must satisfy $\dist(y_{1},{K})\geq c_{0}$, which in particular yields $\vert y_{1}\vert\geq c_{0}$. Since $y_{1}\in B(y,1)$, we have also $\dist(y_{1},\partial {K})\leq c_{0}+1$, and thus \eqref{upper_bound_u} gives
\begin{equation*}
     u(y_{1})\leq C\dist(y_{1},\partial{K})\big\vert y_{1}\vert^{p-1}\leq C(\vert y\vert\vee c_{0})^{p-1},
\end{equation*}
for some constant $C$. Thus, in any case we can choose $y_{1}$ such that \eqref{eq:lowerBoundy1} holds.
\end{proof}

\begin{proof}[Proof of Proposition \ref{uniformAway}]
The proof follows closely the one of \cite[Lem.~20]{DeWa-15}; since we have to take care of the extra condition of uniformity, we choose to rewrite it completely. 

Fix $0< \gamma\leq \frac{\delta}{2(2+\delta)}$, $0<s<1$ and choose $x_{0}\in{\bf R}^{d}$ and $R_{0}>0$ such that $\vert x_{0}\vert =1$, $x_{0}+{K}\subset {K}$ and $\dist(R_{0}x_{0}+{K}, \partial{K})>1$. Let $y\in {K}_{n,\epsilon}^{A}$ and $y^{\pm}:=y\pm R_{0}x_{0}n^{1/2-\gamma}$. Let 
\begin{equation*}
     A_{n}=\bigl\lbrace \sup_{u\leq n}\vert B(u)-S(\lfloor u\rfloor)\vert\leq n^{1/2-\gamma}\bigr\rbrace,
\end{equation*}
where  $B$ is the Brownian motion coupled to $S$ in Section \ref{description_proofs_DW}. By \eqref{coupling_few_moments}, $\pr(A_{n}^{c})\leq n^{-r}$, where $r=\delta/2-2\gamma-\gamma\delta$. Moreover (see the proof of \cite[Lem.~20]{DeWa-15}), for $n$ large enough
\begin{equation*}
     \lbrace \tau_{y}>n\rbrace\cap A_{n}\subset \lbrace \tau^{\bm}_{y^{+}}>n\rbrace\quad\text{and}\quad\lbrace \tau_{y^{-}}^{\bm}>n\rbrace\cap A_{n}\subset \lbrace \tau_{y}>n\rbrace,
\end{equation*}
which yields
\begin{equation}\label{inequalitiesKp}
k_{n}(y^{-})+O(n^{-r})\leq \pr(\tau_{y}>n)\leq k_{n}(y^{+})+O(n^{-r}).
\end{equation}
In order to conclude the first part of the proposition, we have to show that, uniformly on $y\in {K}_{n,\epsilon}^{A}$, $k_{n}(y^{-})\sim k_{n}(y^{+})$ and $n^{-r}=o(k_{n}(y))$ as $n$ goes to infinity. On the one hand, by homogeneity of $k$ and by Proposition \ref{holderContinuity},
\begin{align*}
     \vert k_{n}(y)-k_{n}(y^{\pm})\vert&=\vert \bar{k}(y/\sqrt{n})-\bar{k}(y^{\pm}/\sqrt{n})\vert\\
     &\leq C_{\alpha}(1+\vert y/\sqrt{n}\vert^{p-1})\vert y-y^{\pm}\vert^{\alpha}/n^{\alpha/2}\leq CA^{p-1}n^{-\gamma\alpha}
\end{align*}
for some constant $C>0$. On the other hand, by the estimates of Theorem \ref{estimatek} and Lemma \ref{lowerBoundy1}, there exists a constant $c$ such that 
\begin{equation*}\bar{k}(y)\geq c\frac{u(y)}{u(y_{1})}\geq c\left(1\wedge \frac{u(y)}{(\vert y\vert\vee 1)^{p-1}}\right) \end{equation*}
for $y\in{K}$. Since ${K}$ is convex, by \cite[Lem.~19]{DeWa-15}  there exists a constant $C$ such that
\begin{equation*}
     u(y)\geq C\dist(y,\partial {K})^{p}.
\end{equation*}
Thus, for $y\in {K}_{n,\epsilon}^{A}$ and $n$ large enough,
\begin{equation}\label{lowerBoundk}
\bar{k}(y/\sqrt{n})\geq c\left( 1\wedge \frac{Cn^{-p\epsilon}}{(\vert y/\sqrt{n}\vert\vee 1)^{p-1}}\right)\geq cCn^{-p\epsilon}\left(1\wedge \frac{1}{A^{p-1}}\right)\geq c'n^{-p\epsilon}
\end{equation}
for some $c'>0$. Suppose that $\epsilon$ is such that $(\gamma\alpha\wedge r)> p\epsilon$. Then $n^{-\gamma\alpha}=o(k_{n}(y))$ and $n^{-r}=o(k_{n}(y))$,
which proves that $k_{n}(y)\sim k_{n}(y^{\pm})$ and $n^{-r}=o(k_{n}(y))$. Therefore, \eqref{inequalitiesKp} yields that uniformly in $y\in{K}_{n,\epsilon}^{A}$,
\begin{equation*}\pr(\tau_{y}>n)\underset{n\rightarrow \infty}{\sim} k_{n}(y).\end{equation*}
 
We now turn to the second asymptotics in Proposition \ref{uniformAway}.
Likewise, for $y\in{K}_{n,\epsilon}^{A}$, $x\in{K}\cap B(0,B)$, $t\in [s,1]$ and $D\subset{K}$, we have by \cite[Eq.~(46) and Eq.~(47)]{DeWa-15}
\begin{equation}\label{inequlitiesKpD}
\int_{\sqrt{n}D_{x,t}^{-}}K_{n}(y^{-},z)dz+O(n^{-r})\leq \pr(y+S(n)\in \sqrt{n}D_{x,t},\tau_{y}> n)\leq \int_{\sqrt{n}D_{x,t}^{+}}K_{n}(y^{+},z)dz+O(n^{-r}),
\end{equation}
where 
\begin{equation*}
     \left\{\begin{array}{lll}
     D_{x,t}&=&x+tD,\\
     D_{x,t}^{+}&=&\lbrace z\in{K}: \dist(z,D_{x,t})\leq 2n^{-\gamma}\rbrace,\\
     D_{x,t}^{-}&=&\lbrace z\in D_{x,t}: \dist(z,\partial D_{x,t})\geq 2n^{-\gamma}\rbrace.
     \end{array}\right. 
\end{equation*}     
By homogeneity,
\begin{equation*}\int_{\sqrt{n}D_{x,t}^{\pm}}K_{n}(y^{\pm},z)dz=\int_{D_{x,t}^{\pm}}\bar{K}(y^{\pm}/\sqrt{n},z)dz,\end{equation*}
and by Proposition \ref{holderContinuity}, 
\begin{equation*}\left\vert \bar{K}(y^{\pm}/\sqrt{n},z)-\bar{K}(y/\sqrt{n},z)\right\vert\leq Cn^{-\gamma\alpha}(1+ \vert y/\sqrt{n}\vert)^{p-1}\exp\left(-\frac{\vert y/\sqrt{n}-z\vert^{2}}{2c_{3}}\right),\end{equation*}
which implies
\begin{multline*}
\bigg\vert\int_{D_{x,t}^{\pm}}\bar{K}(y^{\pm}/\sqrt{n},z)dz-\int_{D_{x,t}}\bar{K}(y/\sqrt{n},z)dz\bigg\vert\\
\leq C(1+A)^{p-1}n^{-\gamma\alpha}\int_{D_{x,t}^{+}}\exp\left(-\frac{\vert y/\sqrt{n}-z\vert^{2}}{2c_{3}}\right)dz+C\Vol(D_{x,t}^{+}\setminus D_{x,t}^{-}).
\end{multline*}
The definition of $D_{x,t}^{-}$ and $D_{x,t}^{+}$ yields
\begin{equation*}\Vol(D_{x}^{+}\setminus D_{x}^{-})=\Vol(\{z\in{K}: \dist(z,\partial D_{x,t})< 2n^{-\gamma}\})\leq\Vol(\{z\in{\bf R}^{d}: \dist(z,t\partial D)< 2n^{-\gamma}\}).\end{equation*}
Hence, applying Steiner formula (Theorem 46 in \cite[Chap.\ 16]{Mo-08}) to the convex $tD$ gives
\begin{equation*}\Vol(D_{x,t}^{+}\setminus D_{x,t}^{-})\underset{n\rightarrow \infty}{\sim} 2n^{-\gamma}t^{d-1}\lambda(D)\end{equation*}
uniformly on all $x\in {K}\cap B(0,B)$ and $t\in [s,1]$, where $\lambda(D)$ denotes the surface area of $D$. Thus, since $\alpha\leq 1$, there exists a constant $C$ independent of $t$ such that
\begin{equation*}\left\vert\int_{D_{x,t}^{\pm}}\bar{K}(y^{\pm}/\sqrt{n},z)dz-\int_{D_{x,t}}\bar{K}(y/\sqrt{n},z)dz\right\vert\leq Cn^{-\gamma\alpha}.\end{equation*}
Therefore, \eqref{inequlitiesKpD} yields
\begin{align*}
\int_{D_{x,t}}\bar{K}(y/\sqrt{n},z)dz+O(n^{-r})+O(n^{-\alpha\gamma})&\leq \pr(y+S(n)\in \sqrt{n}D_{x,t},\tau_{y}> n)\\
&\leq \int_{D_{x,t}}\bar{K}(y/\sqrt{n},z)dz+O(n^{-r})+O(n^{-\alpha\gamma}).
\end{align*}
It remains to show that  as $n$ goes to infinity, $n^{-(r\wedge\alpha\gamma)}=o(\int_{D_{x,t}}\bar{K}(y/\sqrt{n},z)dz)$. 
By Theorem \ref{estimateK} and Equation \eqref{lower_bound_ball_C}, we have
\begin{align*}
\bar{K}(y/\sqrt{n},z)&\geq c_{2}\frac{\bar{k}(y/\sqrt{n})\bar{k}(z)}{\sqrt{V(y/\sqrt{n},1)V(z,1)}}\exp\left(-\frac{\vert y/\sqrt{n}-z\vert^{2}}{C_{3}}\right)\\
&\geq c \bar{k}(y/\sqrt{n})\bar{k}(z)\exp\left(-\frac{\vert y/\sqrt{n}-z\vert^{2}}{C_{3}}\right)
\end{align*}
for $y,z\in{K}$ and some constant $c>0$. Hence, if we set $M:=\sup\{\vert z\vert: z\in D\}$, then $D_{x,t}\subset B(0,B+tM)$ and thus
\begin{equation*}\int_{D_{x,t}}\bar{K}(y/\sqrt{n},z)dz\geq C\bar{k}(y/\sqrt{n})\exp(-4(B+tM)^{2}/C_{3})\int_{D_{x,t}}\bar{k}(z)dz.\end{equation*}
Recalling from \eqref{lowerBoundk} that $\bar{k}(z/\sqrt{n})\geq cn^{-p\epsilon}$ for $z\in{K}_{n,\epsilon}$, we get
\begin{equation}\label{lower_bound_K_Dx}\int_{D_{x,t}}\bar{K}(y/\sqrt{n},z)dz\geq c'n^{-p\epsilon}\Vol\left(D_{x,t}\cap\frac{1}{\sqrt{n}}{K}_{n,\epsilon}\right)
\end{equation}
for some constant $c'>0$. Let us find a lower bound for $\Vol\left(D_{x,t}\cap\frac{1}{\sqrt{n}}{K}_{n,\epsilon}\right)$ in \eqref{lower_bound_K_Dx}. Note that since $D_{x,t}$ is contained in $B(0,B+M)$,
\begin{equation*}\Vol\left(D_{x,t}\cap\frac{1}{\sqrt{n}}{K}_{n,\epsilon}\right)\geq \Vol(D_{x,t}\cap {K})-\Vol\left(({K}\cap B(0,B+M))\setminus\frac{1}{\sqrt{n}}{K}_{n,\epsilon}\right).\end{equation*}  
Since $0$ is in the interior of $D$, there exists $u>0$ such that $B(0,u)\subset D$. Hence, by \eqref{lower_bound_ball_C},
\begin{equation*}\Vol(D_{x,t}\cap {K})\geq \Vol\left(B(x,tu)\cap K\right)\geq c(tu)^{d}\geq c(su)^{d}\end{equation*}
for some constant $c$ independent of $x\in B(0,B)\cap{K}$. Since $\frac{1}{\sqrt{n}}{K}_{n,\epsilon}$ converges to $K$ in Hausdorff distance as $n$ goes to infinity,
\begin{equation*}
     \Vol\left(({K}\cap B(0,B+M))\setminus\frac{1}{\sqrt{n}}{K}_{n,\epsilon}\right)
\end{equation*}
goes to $0$ as $n$ goes to infinity. Thus, for $n$ large enough, the quantity above is smaller than $c(su)^{d}/2$ and 
\begin{equation*}
     \Vol\left(D_{x,t}\cap\frac{1}{\sqrt{n}}{K}_{n,\epsilon}\right)\geq \Vol(D_{x,t}\cap {K})-\Vol\left(({K}\cap B(0,B+M))\setminus\frac{1}{\sqrt{n}}{K}_{n,\epsilon}\right)\geq \frac{c(su)^{d}}{2}
\end{equation*}
for all $x\in B(0,B)\cap{K}$, $t\in[s,1]$. The latter inequality together with \eqref{lower_bound_K_Dx} yields $\delta>0$ such that
\begin{equation*}\int_{D_{x,t}}\bar{K}(y/\sqrt{n},z)dz\geq c'\delta n^{-p\epsilon}\end{equation*}
for $n$ large enough, $y\in{K}_{n,\epsilon}^{A}$, $x\in B(0,B)\cap {K}$ and $s\leq t\leq 1$. Choosing $\epsilon< (r\wedge\alpha\gamma)/(2p)$ thus yields that $n^{-(r\wedge\alpha\gamma)}=o(\int_{D_{x,t}}\bar{K}(y/\sqrt{n},z)dz)$. For these values of $\epsilon$, we have 
\begin{equation*} \pr(y+S(n)\in \sqrt{n}D_{x,t},\tau_{y}> n)\sim \int_{D_{x,t}}\bar{K}(y/\sqrt{n},z)dz\end{equation*}
uniformly on $y,x$ and $t$ satisfying the conditions of the statement.
\end{proof}
As a consequence of Proposition \ref{uniformAway}, we give an estimate on the scaling of exit times.
\begin{corollary}\label{boundScalingExitTime}
Let $0<s<1$. There exists a constant $C>0$ only depending on $s$ such that 
\begin{equation*}\limsup_{n\rightarrow \infty}\sup_{\substack{x\in{K}_{n,\epsilon}^{A}\\m\geq sn}}\frac{\pr(\tau_{x}>m)}{\pr(\tau_{x}>n)}\leq C.\end{equation*}
\end{corollary}
\begin{proof}
Let $0<\eta<1$. By Proposition \ref{uniformAway}, there exists $n_{0}\geq 1$ such that for $n\geq n_{0}$ and $x\in{K}_{n,\epsilon}^{A}$,
\begin{equation*}(1-\eta)\bar{k}(x/\sqrt{n})\leq \pr(\tau_{x}>n)\leq (1+\eta)\bar{k}(x/\sqrt{n}).\end{equation*}
Then, for $n>n_{0}/s$ and $m\geq sn$, we have
\begin{equation*}
     \frac{\pr(\tau_{x}>m)}{\pr(\tau_{x}>n)}\leq\frac{1+\eta}{1-\eta}\frac{\bar{k}(x/\sqrt{m})}{\bar{k}(x/\sqrt{n})}.
\end{equation*}
By Lemma \ref{boundScalingk}, there exists a constant $C>0$ only depending on $s$ such that for $m\geq sn$, we have 
$\frac{\bar{k}(x/\sqrt{m})}{\bar{k}(x/\sqrt{n})}\leq C$ for all $x\in{K}$. Therefore,
\begin{equation*}
     \limsup_{n\rightarrow \infty}\sup_{\substack{x\in{K}_{n,\epsilon}^{A}\\m\geq sn}}\frac{\pr(\tau_{x}>m)}{\pr(\tau_{x}>n)}\leq \frac{1+\eta}{1-\eta}C.\qedhere
\end{equation*}
\end{proof}
\subsection{Local limit theorems far from the boundary}
The goal of this subsection is to improve the local limit theorem \cite[Thm~6]{DeWa-15} in order to get a more uniform result, which will correspond to our Theorem \ref{thm:asymp_probab_boundary} restricted to $x\in{K}_{n,\epsilon}^{A}$. As in the previous subsection, the pattern follows closely the one of \cite[Thm~6]{DeWa-15} and is divided into two steps: 
\begin{itemize}
     \item The first step, given by the following proposition, is a more general local limit theorem analogue to \cite[Thm~5]{DeWa-15}:
\end{itemize} 
     \begin{proposition}\label{generalLLT}
Uniformly on $x\in{K}_{n,\epsilon}^{A}$,
\begin{equation*}\sup_{y\in {K}}\left\vert n^{d/2}\pr(x+S(n)=y\vert \tau_{x}>n)-\frac{\bar{K}(x/\sqrt{n},y/\sqrt{n})}{\bar{k}(x/\sqrt{n})}\right\vert\rightarrow 0.\end{equation*} 
\end{proposition}
\begin{itemize}
     \item The second step is a derivation of the more specific local limit theorem analogue to \cite[Thm~6]{DeWa-15}:
\end{itemize}     
\begin{proposition}\label{specificLLT}
Let $x\in {K}$. Then there exists $\kappa>0$ such that uniformly on $y\in{K}^{A}_{n,\epsilon}$, as $n\to\infty$,
\begin{equation*}\pr(x+S(n)=y,\tau_{x}>n)\sim \kappa \cdot V(x)\cdot n^{-d/2-p}\cdot u(y)\cdot\exp\left(-\frac{\vert y\vert^{2}}{2n}\right).\end{equation*}
\end{proposition}

Let us thus first begin with the general local limit theorem given in Proposition~\ref{generalLLT}.
The proof follows the one of \cite[Thm~5]{DeWa-15}, using in addition the heat kernel estimates; since the proof of \cite[Thm~5]{DeWa-15} is already quite technical, we choose to divide the proof of Proposition \ref{generalLLT} into three lemmas. As in \cite{DeWa-15}, we divide ${K}$ into three regions which depend on the choice of $x\in{K}_{n,\epsilon}^{A}$ and two positive parameters $M,\eta>0$ (the dependence does not appear in the notations):
\begin{itemize}
     \item ${K}^{(1)}:=\lbrace y\in{K}: \vert y-x\vert>M\sqrt{n}\rbrace$,
     \item ${K}^{(2)}:=\lbrace y\in{K}: \vert y-x\vert \leq M\sqrt{n},\,\dist(y,\partial {K})\leq2\eta\sqrt{n}\rbrace$, and
     \item ${K}^{(3)}:=\lbrace y\in{K}: \vert y-x\vert \leq M\sqrt{n},\,\dist(y,\partial {K})>2\eta\sqrt{n}\rbrace$.
\end{itemize}
We then estimate the probability $\pr(x+S(n)=y\vert \tau_{x}>n)$ for $y$ belonging to each one of these three regions.
\begin{lemma}\label{estimate_region_1}
One has \begin{equation*}\lim_{M\rightarrow\infty}\limsup_{n\rightarrow\infty}\sup_{x\in {K}_{n,\epsilon}^{A},\,y\in {K}^{(1)}}n^{d/2}\pr(x+S(n)=y\vert \tau_{x}>n)=0.\end{equation*}
\end{lemma}

\begin{lemma}\label{estimate_region_2}
For each $M>0$,
\begin{equation*}\lim_{\eta\rightarrow 0}\limsup_{n\rightarrow\infty}\sup_{x\in {K}_{n,\epsilon}^{A},\,y\in {K}^{(2)}}n^{d/2}\pr(x+S(n)=y\vert \tau_{x}>n)=0.\end{equation*}
\end{lemma}
\begin{lemma}\label{estimate_region_3}
For each $M>0$,
\begin{equation*}\lim_{\eta\rightarrow 0}\limsup_{n\rightarrow\infty}\sup_{x\in {K}_{n,\epsilon}^{A},\,y\in {K}^{(3)}}\left\vert n^{d/2}\pr(x+S(n)=y\vert \tau_{x}>n)-\frac{\bar{K}(x/\sqrt{n},y/\sqrt{n})}{\bar{k}(x/\sqrt{n})}\right\vert=0.\end{equation*}
\end{lemma}

The proofs of Lemmas \ref{estimate_region_1} and \ref{estimate_region_2} follow word for word from the corresponding proofs in the proof of \cite[Thm~5]{DeWa-15}, using Corollary \ref{boundScalingExitTime} to bound uniformly $\frac{\pr(\tau_{x}>m)}{\pr(\tau_{x}>n)}$ for $x\in{K}_{n,\epsilon}^{A}$ and $m\geq sn$. The rest of the subsection is devoted to the proofs of Propositions \ref{generalLLT} and \ref{specificLLT}, as well as of Lemma \ref{estimate_region_3}. 

\begin{proof}[Proof of Lemma \ref{estimate_region_3}]
Set $m=\lfloor \eta^{3}n\rfloor$ and for $y\in{K}^{(3)}$, set 
\begin{equation*}
     {K}_{1}(y)=\lbrace z\in {K}:\vert z-y\vert <\eta \sqrt{n}\rbrace.
\end{equation*}
Then, we can write 
\begin{align*}
&n^{d/2}\pr(x+S(n)=y,\tau_{x}>n)\\
&\quad=\frac{n^{d/2}}{\pr(\tau_{x}>n)}\Biggl(\sum_{z\in {K}\backslash {K}_{1}(y)}\pr(x+S(n-m)=z,\tau_{x}>n-m)\pr(z+S(m)=y,\tau_{z}>m)\\
&\qquad+\sum_{z\in {K}_{1}(y)}\pr(x+S(n-m)=z,\tau_{x}>n-m)\pr(z+S(m)=y)\\
&\qquad-\sum_{z\in {K}_{1}(y)}\pr(x+S(n-m)=z,\tau_{x}>n-m)\pr(z+S(m)=y,\tau_{z}\leq m)\Biggr)\\
&\quad :=\frac{n^{d/2}}{\pr(\tau_{x}>n)}(\Sigma_{1}+\Sigma_{2}-\Sigma_{3}).
\end{align*}
By \eqref{bound_local_probability} with $u=\eta^{-1/2}$, there exist constants $C,a$ independent of $y$ such that for $z\not\in {K}_{1}(y)$ (i.e., for $z\in{\bf R}^{d}$ such that $\vert z-y\vert\geq \eta\sqrt{n}\geq \eta^{-1/2}\sqrt{m}$),
\begin{equation*}\pr(z+S(m)=y,\tau_{z}>m)\leq \pr(z+S(m)=y)\leq Cm^{-d/2}\exp(-a/\eta).\end{equation*}
Thus
\begin{equation*}\frac{n^{d/2}\Sigma_{1}}{\pr(\tau_{x}>n)}\leq C\frac{\pr(\tau_{x}>n-m)}{\pr(\tau_{x}>n)}\eta^{-3d/2}\exp(-a/\eta).\end{equation*}
Since $n-m\geq n/2$ for $\eta<1/2$, by Corollary \ref{boundScalingExitTime} there exists a constant $C'$ independent of $0<\eta<1/2$ such that $\limsup_{n\rightarrow \infty}\sup_{x\in{K}_{n,\epsilon}^{A}}\frac{\pr(\tau_{x}>n-m)}{\pr(\tau_{x}>n)}\leq C',$
which implies
\begin{equation*}\lim_{\eta\rightarrow 0}\limsup_{n\rightarrow \infty}\sup_{x\in{K}_{n,\epsilon}^{A}} C\frac{\pr(\tau_{x}>n-m)}{\pr(\tau_{x}>n)}\eta^{-3d/2}\exp(-a/\eta)\leq C'C\lim_{\eta\rightarrow 0}\eta^{-3d/2}\exp(-a/\eta)=0.\end{equation*}
Hence,
\begin{equation}\label{boundOutsideMiniDisk}\lim_{\eta\rightarrow 0}\limsup_{n\rightarrow \infty}\sup_{y\in{K}^{(3)},\,x\in{K}_{n,\epsilon}^{A}}\frac{n^{d/2}}{\pr(\tau_{x}>n)}\Sigma_{1}=0.
\end{equation}
Likewise, by \cite[Eq.~(76)]{DeWa-15}, there exist constants $a,C$ such that for all $y\in{K}^{(3)}$ and $z\in {K}_{1}(y)$,
\begin{equation*}\pr(z+S(m)=y,\tau_{z}\leq m)\leq Cm^{-d/2}\exp(-a/\eta).\end{equation*}
Therefore, 
\begin{equation*}\frac{n^{d/2}}{\pr(\tau_{x}>n)}\Sigma_{3}\leq C\frac{\pr(\tau_{x}>n-m)}{\pr(\tau_{x}>n)}\eta^{-3d/2}\exp(-a/\eta).\end{equation*}
Applying the same method as for $\Sigma_{1}$ yields 
\begin{equation}\label{boundDirectBorder}\lim_{\eta\rightarrow 0}\limsup_{n\rightarrow \infty}\sup_{y\in{K}^{(3)},\,x\in{K}_{n,\epsilon}^{A}}\frac{n^{d/2}}{\pr(\tau_{x}>n)}\Sigma_{3}=0.
\end{equation}
We now estimate the term $\Sigma_{2}$, which will eventually give the main contribution to the probability $n^{d/2}\pr(x+S(n)=y,\tau_{x}>n)$. By the regular local limit theorem \cite[Prop.~7.9]{Sp-76},
\begin{equation*}
     \sup_{y,z\in{K}}\left\vert (2\pi m)^{d/2}\pr(z+S(m)=y)-\exp\left(-\frac{\vert y-z\vert^{2}}{2m}\right)\right\vert\xrightarrow[m\rightarrow \infty]{}0.\end{equation*}
Hence,
\begin{multline}
\label{expression_Sigma_2}
\frac{1}{\pr(\tau_{x}>n)}\sum_{z\in {K}_{1}(y)}\pr(x+S(n-m)=z,\tau_{x}>n-m)\pr(z+S(m)=y)\\
=\frac{1}{\pr(\tau_{x}>n)}\sum_{z\in {K}_{1}(y)}\pr(x+S(n-m)=z,\tau_{x}>n-m)(2\pi m)^{-d/2}\exp\left(-\frac{\vert y-z\vert^{2}}{2m}\right)\\
+\frac{\pr(\tau_{x}>n-m)}{\pr(\tau_{x}>n)}o(m^{-d/2}),
\end{multline}
where $o(m^{-d/2})$ is uniform on all $x\in{K}_{n,\epsilon}^{A}$ and $y\in{K}^{(3)}$. On the one hand, using Corollary \ref{boundScalingExitTime} yields
\begin{equation}\label{vanishing_little_o}
\limsup_{n\rightarrow \infty}\sup_{x\in{K}_{n,\epsilon}^{A},\,y\in{K}^{(3)}}n^{d/2}o(m^{-d/2})\frac{\pr(\tau_{x}>n-m)}{\pr(\tau_{x}>n)}=0.
\end{equation}
For $n$ large enough, 
\begin{equation*}
     \eta\leq \eta\frac{\sqrt{n}}{\sqrt{n-m}}\leq \frac{\eta}{\sqrt{1-\eta^{3}/2}}, 
\end{equation*}
and the function 
\begin{equation*}
     f_{n}(u)=\left(\frac{n}{2\pi m}\right)^{d/2}\exp\left(-\frac{(n-m)\vert u\vert^{2}}{2m}\right)
\end{equation*}
is uniformly bounded by $({2\pi\eta^{3}})^{-d/2}$ and is Lipschitz continuous with uniform Lipschitz constant $({2\pi\eta^{3}})^{-d/2}{\eta^{-3}}$. Applying the uniform convergence in law from Proposition~\ref{uniformAway} together with \cite[Thm~8.3.2]{Bo-07} to the set of functions $(f_{n})_{n\geq 1}$ and the set of disks of radius $\left(\eta\frac{\sqrt{n}}{\sqrt{n-m}}\right)_{n\geq 1}$ yields
\begin{multline*}
\sup_{\substack{x\in{K}_{n,\epsilon}^{A}\\y\in {K}^{(3)}}}\Bigg\vert\e\left[\mathbf{1}_{B\left(y/\sqrt{n-m},\eta\frac{n}{n-m}\right)}f_{n}\big((x+S(n-m))/\sqrt{n-m}\big) \vert \tau_{x}>n-m\right]\\
\qquad\qquad-\frac{1}{\bar{k}(x/\sqrt{n-m})}\int_{ B\left(y/\sqrt{n-m},\eta\frac{\sqrt{n}}{\sqrt{n-m}}\right)}\bar{K}(x/\sqrt{n-m},z)f_{n}(z)dz\Bigg\vert\xrightarrow[n\rightarrow \infty]{}0.
\end{multline*}
Hence, expanding the expectation in the latter equation, doing the change of variable $u\mapsto \frac{\sqrt{n-m}}{\sqrt{n}}u$ and finally using the scaling property of $K_{t}$ gives
\begin{multline}
\label{equivalent_main_part}
\sup_{\substack{x\in{K}_{n,\epsilon}^{A}\\y\in {K}^{(3)}}}\Bigg\vert\frac{n^{d/2}}{\pr(\tau_{x}>n)}\sum_{z\in {K}_{1}(y)}\pr(x+S(n-m)=z,\tau_{x}>n-m)(2\pi m)^{-d/2}\exp\left(-\frac{\vert y-z\vert^{2}}{2m}\right)\\
-\frac{n^{d}}{\bar{k}\bigl(\frac{x}{\sqrt{n-m}}\bigr)(2\pi(n-m)m)^{d/2}}\int_{\vert u\vert\leq \eta}K_{\frac{n-m}{n}}\left(\frac{x}{\sqrt{n}},\frac{y}{\sqrt{n}}+u\right)\exp\left(-\frac{n\vert u\vert^{2}}{2m}\right)du\Bigg\vert\xrightarrow[n\rightarrow \infty]{}0.
\end{multline}
Consider now the last integral above. As $n$ goes to infinity, $\frac{n^{2}}{m(n-m)}$ converges to $\frac{1}{\eta^{3}(1-\eta^{3})}$ and $\frac{n-m}{n}$ to $1-\eta^{3}$. Lemma \ref{equivalenceKernelTime} yields that 
\begin{multline*}
\frac{1}{\bar{k}\bigl(\frac{x}{\sqrt{n-m}}\bigr)}K_{\frac{n-m}{n}}\left(\frac{x}{\sqrt{n}},\frac{y}{\sqrt{n}}+u\right)\exp\left(-\frac{n\vert u\vert^{2}}{2m}\right)\\
\underset{n\rightarrow \infty}{\sim} \frac{1}{k_{1-\eta^{3}}(x/\sqrt{n})}K_{1-\eta^{3}}\left(\frac{x}{\sqrt{n}},\frac{y}{\sqrt{n}}+u\right)\exp\left(-\frac{\vert u\vert^{2}}{2\eta^{3}}\right)
\end{multline*}
uniformly on $x\in{K}_{n,\epsilon}^{A}$, $y\in {K}^{(3)}$ and $\vert u\vert\leq \eta$. Hence, using \eqref{expression_Sigma_2}, \eqref{vanishing_little_o} and \eqref{equivalent_main_part} yields
\begin{align*}
\limsup_{n\rightarrow \infty}&\sup_{\substack{x\in{K}_{n,\epsilon}^{A}\\ y\in{K}^{(3)}}}\Bigg\vert\frac{n^{d/2}}{\pr(\tau_{x}>n)}\Sigma_{2}\\
-& \frac{(2\pi\eta^{3})^{-d/2}}{(1-\eta^{3})^{d/2}k_{1-\eta^{3}}(x/\sqrt{n})}\int_{\vert u\vert\leq \eta}K_{1-\eta^{3}}\left(\frac{x}{\sqrt{n}},\frac{y}{\sqrt{n}}+u\right)\exp\left(-\frac{\vert u\vert^{2}}{2\eta^{3}}\right)du\Bigg\vert=0.
\end{align*}
Let us show an asymptotic formula for the second term of the latter equation as $\eta$ goes to zero. Let $\theta>0$. By Lemma \ref{equivalenceKernelTime}, there exists $\eta_{0}>0$ such that for all $\eta\leq \eta_{0}$, 
\begin{multline*}
     \frac{(1-\theta)}{\bar{k}(x/\sqrt{n})}\bar{K}\left(\frac{x}{\sqrt{n}},\frac{y}{\sqrt{n}}+u\right)\leq \frac{1}{k_{1-\eta^{3}}(x/\sqrt{n})} K_{1-\eta^{3}}\left(\frac{x}{\sqrt{n}},\frac{y}{\sqrt{n}}+u\right)\\\leq \frac{(1+\theta)}{\bar{k}(x/\sqrt{n})}\bar{K}\left(\frac{x}{\sqrt{n}},\frac{y}{\sqrt{n}}+u\right)
\end{multline*}
for all $x\in{K}_{n,\epsilon}^{A}$, $y\in{K}^{(3)}$, $\vert u\vert\leq\eta$ and $n\geq 1$. Moreover, it comes from Remark~\ref{globallyHolder} that $\frac{1}{\bar{k}(x)}\bar{K}(x,\cdot)$ is H\"older continuous on ${K}\cap B(0,M+\eta)$ with exponent $\alpha$ and some constant $C$ independent of $x\in{K}\cap B(0,A)$, and thus uniformly continuous with uniform continuous bound independent of $x\in{K}\cap B(0,A)$. Hence, since the measure with density 
\begin{equation*}
     \mathbf{1}_{\vert u\vert\leq \eta}(2\pi\eta^{3})^{-d/2}\exp\left(\frac{-\vert u\vert^{2}}{2\eta^{3}}\right)du
\end{equation*}
converges weakly to a Dirac at $0$ as $\eta$ goes to $0$, we have
\begin{multline*}
     \frac{(1-\eta^{3})(2\pi\eta^{3})^{-d/2}}{\bar{k}(x/\sqrt{n})}\int_{\vert u\vert\leq \eta}\bar{K}\left(\frac{x}{\sqrt{n}},\frac{y}{\sqrt{n}}+u\right)\exp\left(-\frac{\vert u\vert^{2}}{2\eta^{3}}\right)du\\\underset{\eta\rightarrow 0}{\sim}\frac{1}{\bar{k}(x/\sqrt{n})}\bar{K}\left(\frac{x}{\sqrt{n}},\frac{y}{\sqrt{n}}\right)
\end{multline*}
uniformly on all $n\geq 1$, $x\in{K}_{n,\epsilon}^{A}$ and $y\in{K}^{(3)}$. Hence, there exists $0<\eta_{1}<\eta_{0}$ such that for $\eta<\eta_{1}$,
\begin{align*}
\frac{(2\pi\eta^{3})^{-d/2}}{(1-\eta^{3})k_{1-\eta^{3}}(x/\sqrt{n})}&\int_{\vert u\vert\leq \eta}K_{1-\eta^{3}}\left(\frac{x}{\sqrt{n}},\frac{y}{\sqrt{n}}+u\right)\exp\left(-\frac{\vert u\vert^{2}}{2\eta^{3}}\right)du\\
&\leq (1+\theta) \frac{(2\pi\eta^{3})^{-d/2}}{\bar{k}(x/\sqrt{n})}\int_{\vert u\vert\leq \eta}\bar{K}\left(\frac{x}{\sqrt{n}},\frac{y}{\sqrt{n}}+u\right)\exp\left(-\frac{\vert u\vert^{2}}{2\eta^{3}}\right)du\\
&\leq (1+\theta)^{2}\frac{1}{\bar{k}(x/\sqrt{n})}\bar{K}\left(\frac{x}{\sqrt{n}},\frac{y}{\sqrt{n}}\right),
\end{align*}
and similarly,
\begin{multline*}
     \frac{(2\pi\eta^{3})^{-d/2}}{(1-\eta^{3})k_{1-\eta^{3}}(x/\sqrt{n})}\int_{\vert u\vert\leq \eta}K_{1-\eta^{3}}\left(\frac{x}{\sqrt{n}},\frac{y}{\sqrt{n}}+u\right)\exp\left(-\frac{\vert u\vert^{2}}{2\eta^{3}}\right)du\\\geq \frac{(1-\theta)^{2}}{\bar{k}(x/\sqrt{n})}\bar{K}\left(\frac{x}{\sqrt{n}},\frac{y}{\sqrt{n}}\right)
\end{multline*}
for all $n\geq 1$, $x\in{K}_{n,\epsilon}^{A}$ and $y\in{K}^{(3)}$. Thus, 
\begin{multline*}
     \lim_{\eta\rightarrow 0}\sup_{\substack{n\geq 1,\,x\in{K}_{n,\epsilon}^{A}\\y\in{K}^{(3)}}}\Bigg\vert \frac{(2\pi\eta^{3})^{-d/2}}{(1-\eta^{3})k_{1-\eta^{3}}(x/\sqrt{n})}\int_{\vert u\vert\leq \eta}K_{1-\eta^{3}}\left(\frac{x}{\sqrt{n}},\frac{y}{\sqrt{n}}+u\right)\exp\left(-\frac{\vert u\vert^{2}}{2\eta^{3}}\right)du\\
-\frac{1}{\bar{k}(x/\sqrt{n})}\bar{K}\left(\frac{x}{\sqrt{n}},\frac{y}{\sqrt{n}}\right)\Bigg\vert=0.
\end{multline*}
This yields finally
\begin{align*}
&\lim_{\eta\rightarrow 0}\limsup_{n\rightarrow \infty}\sup_{\substack{x\in{K}_{n,\epsilon}^{A}\\ y\in{K}^{(3)}}}\left\vert\frac{n^{d/2}}{\pr(\tau_{x}>n)}\Sigma_{2}-\frac{1}{\bar{k}(x/\sqrt{n})}\bar{K}\left(\frac{x}{\sqrt{n}},\frac{y}{\sqrt{n}}\right)\right\vert\\
&\leq\lim_{\eta\rightarrow 0}\limsup_{n\rightarrow \infty}\sup_{\substack{x\in{K}_{n,\epsilon}^{A}\\ y\in{K}^{(3)}}}\Bigg\vert\frac{n^{d/2}}{\pr(\tau_{x}>n)}\Sigma_{2}\\
&\qquad\qquad- \frac{(2\pi\eta^{3})^{-d/2}}{(1-\eta^{3})^{d/2}k_{1-\eta^{3}}(x/\sqrt{n})}\int_{\vert u\vert\leq \eta}K_{1-\eta^{3}}\left(\frac{x}{\sqrt{n}},\frac{y}{\sqrt{n}}+u\right)\exp\left(-\frac{\vert u\vert^{2}}{2\eta^{3}}\right)du\Bigg\vert\\
&+\lim_{\eta\rightarrow 0}\limsup_{n\rightarrow \infty}\sup_{\substack{x\in{K}_{n,\epsilon}^{A}\\y\in{K}^{(3)}}}\Bigg\vert\frac{(2\pi\eta^{3})^{-d/2}}{(1-\eta^{3})^{d/2}k_{1-\eta^{3}}(x/\sqrt{n})}\int_{\vert u\vert\leq \eta}K_{1-\eta^{3}}\left(\frac{x}{\sqrt{n}},\frac{y}{\sqrt{n}}+u\right)\exp\left(-\frac{\vert u\vert^{2}}{2\eta^{3}}\right)du\\
&\qquad\qquad-\frac{1}{\bar{k}(x/\sqrt{n})}\bar{K}\left(\frac{x}{\sqrt{n}},\frac{y}{\sqrt{n}}\right)\Bigg\vert=0.\\
\end{align*}
Combining the latter equation with \eqref{boundOutsideMiniDisk} and \eqref{boundDirectBorder}, we get
\begin{equation*}
     \lim_{\eta\rightarrow 0}\limsup_{n\rightarrow \infty}\sup_{y\in{K}^{(3)},\,x\in{K}_{n,\epsilon}^{A}}\left\vert n^{d/2}\pr(x+S(n)=y\vert \tau_{x}>n)-\frac{\bar{K}\left(x/\sqrt{n},y/\sqrt{n}\right)}{\bar{k}(x/\sqrt{n})}\right\vert=0.\qedhere
\end{equation*}
\end{proof}

With the help of Lemmas \ref{estimate_region_1}, \ref{estimate_region_2} and \ref{estimate_region_3}, we can achieve the proof of Proposition \ref{generalLLT}.
\begin{proof}[Proof of Proposition \ref{generalLLT}]
Note first that by Theorem \ref{estimateK}, we have
\begin{equation}
\label{bound_K_away}
     \sup_{\substack{x,y\in {K}\\ \vert y-x\vert \geq M}}\frac{\bar{K}(x,y)}{\bar{k}(x)}\xrightarrow[M\rightarrow \infty]{} 0.
\end{equation}
By Theorem \ref{estimatek}, $\bar{k}(y)\leq C_{1}\frac{u(y)}{u(y_{1})}$ for $y\in{K}$, and by \eqref{lower_bound_y_1}, there exists a constant $c>0$ such that $u(y_{1})>c$ for $y\in{K}$. Thus, there exists a constant $C$ such that $\bar{k}(y)\leq C u(y)$ for all $y\in{K}$. Hence, by \eqref{upper_bound_u} in Appendix \ref{appendix_heat_estimates}, there exists a constant $C'$ such that 
\begin{equation*}\bar{k}(y)\leq C'\dist(y,\partial {K})\vert y\vert^{p-1},\end{equation*}
which yields that for each $M>0$
\begin{equation*}\sup_{\substack{y\in {K},\,\vert y\vert \leq A+M\\ \dist(y,\partial {K})<\eta}}\bar{k}(y)\xrightarrow[\eta\rightarrow 0]{} 0.\end{equation*}
Therefore, Theorem \ref{estimateK} implies that for each $M>0$,
\begin{equation}\label{bound_K_edge}
\sup_{\substack{y\in {K},\,\vert x-y\vert \leq M\\ \dist(y,\partial {K})<\eta}}\frac{\bar{K}(x,y)}{\bar{k}(x)}\xrightarrow[\eta\rightarrow 0]{} 0,
\end{equation}
uniformly in $x\in {K}$, $\vert x\vert\leq A$. Let $\delta>0$. By Lemma \ref{estimate_region_1} and \eqref{bound_K_away}, there exists $M>0$ such that
\begin{equation*}
     \limsup_{n\rightarrow \infty}\sup_{\substack{x\in{K}_{n,\epsilon}^{A}\\y\in{K}^{(1)}}}n^{d/2}\pr(x+S(n)=y\vert \tau_{x}>n)\leq \delta
\end{equation*}
and
\begin{equation*}\sup_{\substack{x\in{K}_{n,\epsilon}^{A}\\y\in {K}^{(1)}}}\frac{\bar{K}(x/\sqrt{n},y/\sqrt{n})}{\bar{k}(x/\sqrt{n})}\leq\delta,\end{equation*}
where we recall that ${K}^{(1)}=\lbrace y\in{K}: \vert y-x\vert> M\sqrt{n}\rbrace$. Hence, for this value of $M$,
\begin{equation*}\limsup_{n\rightarrow \infty}\sup_{\substack{x\in{K}_{n,\epsilon}^{A}\\y\in{K}^{(1)}}}\bigg\vert n^{d/2}\pr(x+S(n)=y\vert \tau_{x}>n)- \frac{\bar{K}(x/\sqrt{n},y/\sqrt{n})}{\bar{k}(x/\sqrt{n})}\bigg\vert \leq 2\delta.\end{equation*}
Then, by Lemma \ref{estimate_region_2}, \eqref{bound_K_edge} and Lemma \ref{estimate_region_3}, there exists $\eta>0$ such that defining ${K}^{(2)}$ and ${K}^{(3)}$ with this value of $\eta$ and the value of $M$ chosen above gives 
\begin{equation*}\limsup_{n\rightarrow \infty}\sup_{\substack{x\in{K}_{n,\epsilon}^{A}\\y\in{K}^{(2)}}}n^{d/2}\pr(x+S(n)=y\vert \tau_{x}>n)\leq \delta,\qquad \sup_{\substack{x\in{K}_{n,\epsilon}^{A}\\y\in {K}^{(2)}}}\frac{\bar{K}(x/\sqrt{n},y/\sqrt{n})}{\bar{k}(x/\sqrt{n})}\leq\delta,\end{equation*}
and 
\begin{equation*}\limsup_{n\rightarrow \infty}\sup_{\substack{x\in{K}_{n,\epsilon}^{A}\\y\in{K}^{(3)}}}\bigg\vert n^{d/2}\pr(x+S(n)=y\vert \tau_{x}>n)- \frac{\bar{K}(x/\sqrt{n},y/\sqrt{n})}{\bar{k}(x/\sqrt{n})}\bigg\vert \leq\delta.\end{equation*}
Hence, for these values of $M$ and $\eta$, there exists $n_{0}$ such that for $n\geq n_{0}$
\begin{equation*}\bigg\vert n^{d/2}\pr(x+S(n)=y\vert \tau_{x}>n)- \frac{\bar{K}(x/\sqrt{n},y/\sqrt{n})}{\bar{k}(x/\sqrt{n})}\bigg\vert \leq 3\delta\end{equation*}
for all $x\in{K}_{n,\epsilon}^{A}$ and $y\in{K}^{(1)}\cup {K}^{(2)}\cup {K}^{(3)}={K}$.
\end{proof}

We now turn to the proof of Proposition \ref{specificLLT}, which gives the exact asymptotics of the probability $\pr(x+S(n)=y,\tau_{x}>n)$ for $y$ varying with $n$. Proposition \ref{specificLLT} is an extension of \cite[Thm~6]{DeWa-15}. We notice that the hypothesis of Proposition \ref{specificLLT} stating that $y$ has to remain in ${K}_{n,\epsilon}^{A}$ will be removed in the next subsection, where the boundary case will be considered.

\begin{proof}[Proof of Proposition \ref{specificLLT}]
Set $m=\lfloor n/2\rfloor$. Classically, 
\begin{equation*}
     \pr(x+S(n)=y,\tau_{x}>n)=\sum_{z\in{K}}\pr(x+S(n-m)=z,\tau_{x}>n-m)\pr(y+S'(m)=z,\tau'_{y}>m).
\end{equation*}
Let $B>0$. On the one hand, by \eqref{bound_local_probability_in_cone}, there exists a constant $C(x)$ such that
\begin{align*}
     \Sigma_{1}(B,n)&:=\sum_{z\in{K},\,\vert z\vert >(A+B)\sqrt{n}}\pr(x+S(n-m)=z,\tau_{x}>n-m)\pr(y+S'(m)=z,\tau'_{y}>m)\\
     &\leq C(x)2^{p/2+d/2}n^{-p/2-d/2}\pr(y+S'(m)>(B+A)\sqrt{n}, \tau'_{y}>m).
\end{align*}
Applying Proposition \ref{uniformAway} to the convex set $B(0,\sqrt{2}(A+B))$ and $x=0$ yields
\begin{align*}
\pr(y+S'(m)&>(B+A)\sqrt{n}, \tau'_{y}>m)\\&\quad=\pr(\tau'_y>m)-\pr(y+S'(m)\leq(B+A)\sqrt{n}, \tau'_{y}>m)\\
&\quad\underset{n\rightarrow \infty}{\sim}\bar{k}(y/\sqrt{m})-\int_{\vert w\vert\leq \sqrt{2}(A+B)}\bar{K}(y/\sqrt{m},w)dw\\
&\quad\underset{n\rightarrow \infty}{\sim}\int_{\vert w\vert> \sqrt{2}(A+B)}\bar{K}(y/\sqrt{m},w)dw,
\end{align*}
where we have used $\int_{K}\bar{K}(y/\sqrt{m},w)dw=\bar{k}(y/\sqrt{m})$ in the last equivalence. Hence,
\begin{equation*}\Sigma_{1}(B,n)\leq Cn^{-p/2-d/2}\int_{\vert w\vert> \sqrt{2}(A+B)}\bar{K}(y/\sqrt{m},w)dw\end{equation*}
for some constant $C>0$, and by Theorem \ref{estimateK} and Corollary \ref{boundScalingExitTime},
\begin{equation}
\label{errorTerm1}
     \lim_{B\rightarrow \infty}\lim_{n\rightarrow \infty}\sup_{y\in {K}_{n,\epsilon}^{A}}\frac{n^{p/2+d/2}}{\bar{k}(y/\sqrt{n})}\Sigma_{1}(B,n)\leq C' \lim_{B\rightarrow \infty}\lim_{n\rightarrow \infty}\sup_{y\in {K}_{n,\epsilon}^{A}}\frac{n^{p/2+d/2}}{\bar{k}(y/\sqrt{m})}\Sigma_{1}(B,n)=0.
\end{equation}
On the other hand, by applying Proposition \ref{generalLLT} to $S'(m)$ and \cite[Thm~5]{DeWa-15} to $S(n-m)$, we obtain that
\begin{align*}
&\sum_{z\in{K},\,\vert z\vert \leq(A+B)\sqrt{n}}\pr(x+S(n-m)=z,\tau_{x}>n-m)\pr(y+S'(m)=z,\tau'_{y}>m)\\
&=\kappa 2^{p+d/2}V(x)n^{-p/2-d}\sum_{z\in{K},\,\vert z\vert \leq (A+B)\sqrt{n}}u(\sqrt{2}z/\sqrt{n})\bar{K}(\sqrt{2}y/\sqrt{n},\sqrt{2}z/\sqrt{n})\exp\left(-\frac{\vert z\vert^{2}}{n}\right)+o(R_{n})\\
&:=\Sigma_{2}(B,n)+o(R_{n}),
\end{align*}
with $\kappa$ being a constant coming from \cite[Thm~5]{DeWa-15} and
\begin{equation*}R_{n}=\pr(\tau_{x}>n-m)n^{-d/2}\bar{k}(y/\sqrt{m})+(n-m)^{-p/2-d/2}k(y/\sqrt{m}).\end{equation*}
Using \eqref{asymptotic_survival_time}, we get
\begin{equation}\label{errorTerm2}
o(R_{n})=o(n^{-p/2-d/2}\bar{k}(y/\sqrt{m})).
\end{equation}
Moreover, as $n$ goes to infinity, by the H\"older continuity of $K(z,\cdot)$ on ${K}\cap B(0,A+B)$, we have the uniform convergence of the Riemann integral
\begin{multline*}
\sup_{y\in {K}_{n,\epsilon}^{A}}\Bigg\vert n^{-d/2}\sum_{z\in{K},\,\vert z\vert \leq (A+B)\sqrt{n}}u(\sqrt{2}z/\sqrt{n})\bar{K}(\sqrt{2}y/\sqrt{n},\sqrt{2}z/\sqrt{n})\exp\left(-\frac{\vert z\vert^{2}}{n}\right)-\\
2^{p/2}\int_{\vert w\vert \leq A+B}u(w)\bar{K}(\sqrt{2}y/\sqrt{n},\sqrt{2}w)e^{-\vert w\vert^{2}}dw\Bigg\vert\xrightarrow[n\rightarrow \infty]{}0,
\end{multline*}
which yields 
\begin{equation}
\label{mainTerm}
     \Sigma_{2}(B,n)\underset{n\rightarrow \infty}{\sim}\kappa 2^{p+d}V(x)n^{-p/2-d/2}\int_{\vert w\vert \leq A+B}u(w)\bar{K}(\sqrt{2}y/\sqrt{n},\sqrt{2}w)e^{-\vert w\vert^{2}}dw
\end{equation}
uniformly on $y\in{K}^{A}_{n,\epsilon}$. In particular, by \cref{estimateK}, there exist $C>0$ and $n_{0}\geq 1$ such that 
\begin{equation*}
     \Sigma_{2}(B,n)\geq C n^{-p/2-d/2}\bar{k}(\sqrt{2}y/\sqrt{n})
\end{equation*}
for all $n\geq n_{0}$ and $y\in{K}^{A}_{n,\epsilon}$. Thus, combining \eqref{errorTerm1}, \eqref{errorTerm2} and \eqref{mainTerm} yields
\begin{multline*}
     \pr(x+S(n)=y,\tau_{x}>n)\\\underset{n\rightarrow \infty}{\sim}\kappa 2^{p+d/2}V(x)n^{-p/2-d/2}\int_{w\in {K}}u(w)\bar{K}(\sqrt{2}y/\sqrt{n},\sqrt{2}w)e^{-\vert w\vert^{2}}dw.
\end{multline*}
It remains to compute the above integral. Note first that by homogeneity of $K$, 
\begin{equation*}
     \bar{K}(\sqrt{2}y/\sqrt{n},\sqrt{2}w)=K_{1/2}(y/\sqrt{n},w).
\end{equation*}
Moreover, by \cite[Lem.~18]{DeWa-15}, 
 \begin{equation}\label{asymptotics_K_zero}
 K_{t}(x,w)\underset{\substack{x\rightarrow 0\\\vert w\vert\leq \vert x\vert^{-1/2}}}{\sim} \kappa' u(x)u(w)\exp\left(-\frac{\vert w\vert ^{2}}{2t}\right)t^{-p-d/2}
 \end{equation}
for some constant $\kappa'$. Hence,
\begin{multline*}
     \int_{w\in {K}}u(w)K_{1/2}(y/\sqrt{n},w)e^{-\vert w\vert^{2}}dw\\=\lim_{x\rightarrow 0}\frac{1}{\kappa' u(x)2^{p+d/2}}\int_{w\in{K}}K_{1/2}(x,w)K_{1/2}(y/\sqrt{n},w)dw.
\end{multline*}
Since $K_{t}$ is symmetric and is a transition kernel,
\begin{equation*}\int_{w\in {K}}u(w)K_{1/2}(y/\sqrt{n},w)e^{-\vert w\vert^{2}}dw=\lim_{x\rightarrow 0}\frac{1}{\kappa' u(x)2^{p+d/2}}\bar{K}(x,y/\sqrt{n}).\end{equation*}
Finally, using \eqref{asymptotics_K_zero} again in the latter equality yields
\begin{equation*}\int_{w\in {K}}u(w)K_{1/2}(y/\sqrt{n},w)e^{-\vert w\vert^{2}}dw=u(y/\sqrt{n})2^{-p-d/2}\exp\left(-\frac{\vert y\vert^{2}}{2n}\right)\end{equation*}
and 
\begin{equation*}
     \pr(x+S(n)=y,\tau_{x}>n)\underset{n\rightarrow \infty}{\sim}\kappa V(x)n^{-p/2-d/2}u(y/\sqrt{n})\exp\left(-\frac{\vert y\vert^{2}}{2n}\right).\qedhere
\end{equation*}
\end{proof}

\subsection{Local limit theorem close to the boundary}
Let us first give a uniform lower bound on the survival probability for an asymptotically strongly irreducible random walk in a cone. Although this result will be applied to the reversed random walk $S'$ in the sequel, we state it in full generality for potential further applications.
\begin{lemma}\label{lowerBoundProbSurv}
Let $K$ be a convex cone with exponent $p$. Let $S$ be a random walk which is asymptotically strongly irreducible in $K$ (see \ref{H:strongly_irreducible}) and has increments with moments of order $p$ and $2+\epsilon$, $\epsilon>0$. Then there exists $c>0$ such that for all $z\in{K}\cap \Lambda$ with $\vert z\vert$ large enough and all $n\geq 1$,
\begin{equation*}\pr(\tau_{z}>n)\geq cn^{-p/2}.\end{equation*}
\end{lemma}
Let us first briefly explain the idea of the proof. By \cite[Thm~1]{DeWa-15}, we know that for any fixed point $v$ inside the cone $K$ and far enough from the boundary, then as $n$ goes to infinity $\mathbf{P}(\tau_{v}>n)\sim V(v)n^{-p/2}$, where $V(v)>0$. The goal is then to compare $\mathbf{P}(\tau_{z}>n)$ with $\mathbf{P}(\tau_v >n)$, where $v$ (resp.\ $z$) is as above (resp.\ as in the statement of Lemma \ref{lowerBoundProbSurv}). This comparison is done in two steps:
\begin{itemize}
     \item The first step is rather straightforward and consists in achieving the comparison for points $z$ which are also far enough from the boundary.
     \item The second step is to get the comparison for points $z$ arbitrary close to the boundary. The hypothesis of strong irreducibility is crucial for this second step, and it is the only place in the paper where this property is used.
\end{itemize}
Note also that the statement of this lemma and its proof could be adapted to the case of a $\mathcal C^2$ and star-shaped cone.
\begin{proof}
Let $v\in \Lambda$ be an arbitrary lattice point in the interior of ${K}$. By \cite[Thm~1]{DeWa-15} (see also \eqref{asymptotic_survival_time} in this paper), for any $z\in K$, $\mathbf{P}(\tau_{z}>n)\sim \chi V(z)n^{-p/2}$ as $n$ goes to infinity, where $V:K\to [0,\infty)$ is harmonic for the random walk $S$ killed outside of ${K}$. The map $V$ can vanish on $v$, but by \cite[Lem.~13 (a)]{DeWa-15}, there exists $t>0$ such that $V(tv)>0$. Hence, for this $t>0$, there exists $c>0$ such that for all $n\geq 1$,
\begin{equation}
\label{minimum_tv}
     \pr\left(\tau_{t v}>n\right)\geq cn^{-p/2}.
\end{equation}

We first give a similar lower bound for any point of $K$ far enough from the boundary (step 1 above). For $z\in{K}$, let us denote by $d_{v}(z)$ the vertical distance from $z$ to $\partial{K}$ parallel to $v$, namely,
\begin{equation*}
     d_{v}(z)=\sup\{\lambda>0: z-\lambda v\in K\}.
\end{equation*}
See Figure \ref{fig:explanation}. Since $K$ is open, one has $d_{v}(z)> 0$ for any $z\in {K}$ and $d_{v}(tv)=t$ for $t>0$. Moreover, if $y\in z+K$, then $y-d_{v}(z)v\in z-d_{v}(z)v+ K\subset K$, which yields
\begin{equation}\label{translation_v_distance}
d_{v}(y)\geq d_{v}(z)+\sup\{\lambda>0: y-\lambda v\in z+K\}.
\end{equation}
When $C$ is a cone distinct from $K$ containing $x$, denote by $\tau_{x,C}$ the exit time from $C$ for the random walk $x+S$. Suppose that $z\in{K}$ is such that $d_{v}(z)>t$. Since $K$ is convex, $z-tv\in {K}$ and thus $z-tv+{K}\subset K$. Moreover, the obvious equality $z=z-tv+tv$ implies that $z\in z-tv+K$. Hence \eqref{minimum_tv} yields
\begin{equation}
\label{lower_bound_far}
     \pr(\tau_{z}>n)\geq \pr\left(\tau_{z,z-tv+K}>n\right)\geq cn^{-p/2}
\end{equation}
for all $n\geq 1$, since by translation $\pr(\tau_{z,z-tv +K}>n)=\pr(\tau_{t v,K}>n)$.

Let us now deal with points $z$ which are close to the boundary of $K$ (step 2 above).
Let $z\in K\cap\Lambda$ with $\vert z\vert\geq R$, and let $\Gamma_{1}$ be a path in $B(z,R)\cap {K}$ from $z$ to $z+K$ with positive probability for $S$, whose existence is guaranteed by \ref{H:strongly_irreducible}. Let $n_{0}$ be the maximum number of points of the lattice inside $B(z,R)$ for $z\in{\bf R}^{d}$ and set 
\begin{equation*}
     \eta:=\min_{\substack{\gamma\subset B(0,R),\,l(\gamma)\leq n_{0}\\\pr(\gamma)>0}} \pr(\gamma),
\end{equation*}     
$l(\gamma)$ denoting the length of $\gamma$, where the probability is understood with respect to $S$. Since $\{\gamma\subset B(0,R):\,l(\gamma)\leq n_{0}, \mathbf{P}(\gamma)>0\}$ is finite and non-empty, $\eta>0$. Moreover, set 
\begin{equation*}
     \delta:=\min\{d_{v}(z): z\in \Lambda\cap K, \vert z\vert\leq R\}.
\end{equation*}
Note that $\delta>0$ since $K$ is open and $B(0,R)\cap K$ is finite. Since there are at most $n_{0}$ points of the lattice in $B(z,R)$ we may assume $l(\Gamma_{1})\leq n_{0}$, which yields $\pr(\Gamma_{1})>\eta$. Let $z_{1}$ be the endpoint of $\Gamma_{1}$. Since $z_{1}\in z+K\cap \Lambda$ and $z\in \Lambda$, then
\begin{equation*}
     \sup\{\lambda>0: z_{1}-\lambda v\in z+K\}\geq \delta.
\end{equation*}
Hence, \eqref{translation_v_distance} yields  
\begin{equation*}
     d_{v}(z_{1})>d_{v}(z)+\delta.
\end{equation*}
Repeating the operation for $z_{1},z_{2},\ldots$ at most $\lfloor \frac{t}{\delta}\rfloor$ times yields ultimately a path $\Gamma$ in $K$ from $z$ to $z'\in K$ of length $m$ such that $d_{v}(z')\geq t$ and $\pr(\Gamma)\geq \eta^{\lfloor \frac{t}{\delta}\rfloor}$ (see Figure \ref{fig:explanation} for a picture of the construction). 
\begin{figure}
\scalebox{0.5}{{
\fontsize{14pt}{11pt}\selectfont
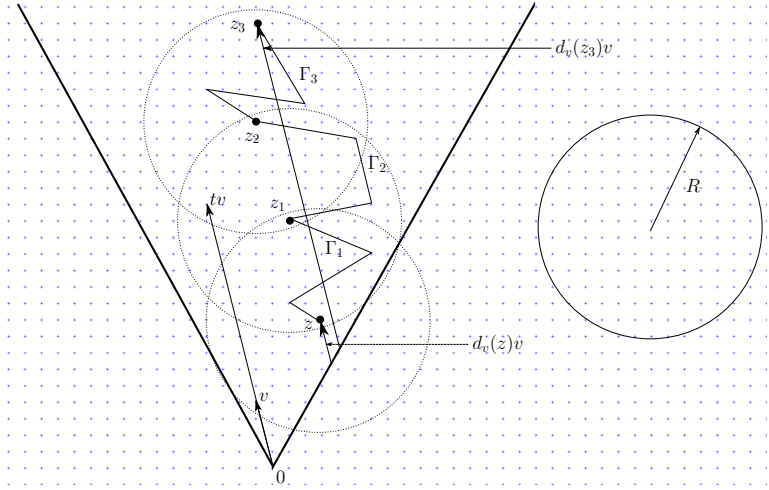}}
\caption{Construction of the path $\Gamma=\Gamma_{1}\cdot\Gamma_{2}\cdot\Gamma_3$.}
\label{fig:explanation}
\end{figure}
Then, the Markov property and \eqref{lower_bound_far} yield
\begin{equation*}
     \pr(\tau_{z}>n)\geq \eta^{\lfloor \frac{t}{\delta}\rfloor}\pr(\tau_{z'}>n-m)\geq\eta^{\lfloor \frac{t}{\delta}\rfloor} c(n-m)^{-p/2}\geq c'n^{-p/2},
\end{equation*}
for some constant $c'$ independent of $z$.
\end{proof}

We are now able to prove Theorem \ref{thm:asymp_probab_boundary}. Recall that $t'_{y,\epsilon}(n)$ denotes the first time that $\{y+S'(m)\}$ reaches ${K}_{n,\epsilon}$, see \eqref{eq:def_t_x_n}, and $y'_{\epsilon}(n):=y+S'(t'_{y,\epsilon}(n))$.

\begin{proof}[Proof of Theorem \ref{thm:asymp_probab_boundary}]
Suppose that $y$ goes to infinity with $y/\vert y\vert$ converging to $\sigma\in \partial \Sigma$. Since $\pr(x+S(n)=y,\tau_{x}>n)=\pr(y+S'(n)=x,\tau'_{y}>n)$, one has
\begin{align}
\nonumber\pr&(x+S(n)=y,\tau_{x}>n)\\&\label{eq:term1_3}=\pr(y+S'(n)=x,t'_{y,\epsilon}(n)>n^{1-\epsilon},\tau'_{y}>n)\\
&\label{eq:term2_3}+\pr(y+S'(n)=x,t'_{y,\epsilon}(n)\leq n^{1-\epsilon},\vert S'(t'_{y,\epsilon}(n))\vert>n^{1/2-\epsilon/8},\tau'_{y}>n)\\
&\label{eq:term3_3}+\pr(y+S'(n)=x,t'_{y,\epsilon}(n)\leq n^{1-\epsilon},\vert S'(t'_{y,\epsilon}(n))\vert\leq n^{1/2-\epsilon/8},\tau'_{y}>n).
\end{align}
By \eqref{bound_stopping_time}, the term \eqref{eq:term1_3} above satisfies
\begin{equation*}
     \pr(y+S'(n)=x,t'_{y,\epsilon}(n)> n^{1-\epsilon},\tau'_{y}>n)\leq \pr(t'_{y,\epsilon}(n)> n^{1-\epsilon},\tau'_{y}>n^{1-\epsilon})\leq \exp(-Cn^{\epsilon})
\end{equation*}
for some constant $C>0$. We now focus on \eqref{eq:term2_3}. Moreover, applying Lemma \ref{Lemma_24_revisited} with $\alpha=0$ and $s=r$ for some $r>p+2$ such that ${\bf E}(\vert X\vert^r)<+\infty$, we get
\begin{multline}
\pr(t'_{y,\epsilon}(n)\leq n^{1-\epsilon},\vert S'(t'_{y,\epsilon}(n))\vert>n^{1/2-\epsilon/8},\tau'_{y}>n)\\
\leq \pr(\sup_{1\leq \ell\leq n^{1-\epsilon}}\vert S'(\ell)\vert>n^{1/2-\epsilon/8},\tau'_{y}>\ell)\leq Cn^{-(r-2)/2}.\label{small_y'}
\end{multline}
Thus, \eqref{bound_local_probability_in_cone} yields 
\begin{align*}
\pr&(y+S'(n)=x,t'_{y,\epsilon}(n)\leq n^{1-\epsilon},\vert S(t'_{y,\epsilon}(n))\vert> n^{1/2-\epsilon/8},\tau'_{y}>n)\\
&\leq \pr(t'_{y,\epsilon}(n)\leq n^{1-\epsilon},\vert S'(t'_{y,\epsilon}(n))\vert>n^{1/2-\epsilon/8},\tau'_{y}>n)\sup_{\substack{z\in {K}\\ n-n^{1-\epsilon}\leq m\leq n}}\pr(x+S(m)=z,\tau_{x}>m)\\
&\leq C'n^{-(r-2)/2-p/2-d/2},
\end{align*}
for some constant $C'>0$. We now look at the term \eqref{eq:term3_3}.
By Proposition \ref{specificLLT} and the strong Markov property of $S'$,
\begin{align*}
\pr(y+S'(n)=x&,t'_{y,\epsilon}(n)\leq n^{1-\epsilon},\vert S(t'_{y,\epsilon}(n))\vert\leq n^{1/2-\epsilon/8},\tau'_{y}>n)\\
&=\e[\pr(x+S(n-t'_{y,\epsilon}(n))=y'_{\epsilon}(n);\tau_{x}>n-t'_{y,\epsilon}(n));\\&\hspace{3cm}t'_{y,\epsilon}(n)\leq n^{1-\epsilon},\vert S(t'_{y,\epsilon}(n))\vert\leq n^{1/2-\epsilon/8}, \tau'_{y}>t'_{y,\epsilon}(n)]\\
&\sim \kappa V(x)n^{-p-d/2}\e[u(y'_{\epsilon}(n))\exp(-\vert y\vert^{2}/(2n));t'_{y,\epsilon}(n)\leq n^{1-\epsilon},\\
&\hspace{5.53cm}\vert S(t'_{y,\epsilon}(n))\vert\leq n^{1/2-\epsilon/8}, \tau'_{y}>t'_{y,\epsilon}(n)]
\end{align*}
uniformly on $y\in{K}$ with $y\leq A\sqrt{n}$. Using the definition \eqref{eq:def_t_x_n} of $t'_{y,\epsilon}(n)$ and the lower bound $u(z)\geq c\dist(z,\partial{K})^{p}$ given by \cite[Lem.~19]{DeWa-15} (see also \eqref{lower_bound_y_1}), we get
\begin{equation*}u(y'_{\epsilon}(n))\exp(-\vert y\vert^{2}/(2n))\geq cn^{O(\epsilon)}\end{equation*}
on the event $\lbrace t'_{y,\epsilon}(n)\leq n^{1-\epsilon}, \tau'_{y}>t'_{y,\epsilon}(n),\vert S(t'_{y,\epsilon}(n))\vert\leq n^{1/2-\epsilon/8}\rbrace$. Hence,
there exists $c'$ such that for $n$ large enough,
\begin{multline*}
\pr(y+S'(n)=x,t'_{y,\epsilon}(n)\leq n^{1-\epsilon},\tau'_{y}>t'_{y,\epsilon}(n))\\
\geq c'n^{-p-d/2+O(\epsilon)}\pr(t'_{y,\epsilon}(n)\leq n^{1-\epsilon}, \tau'_{y}>t'_{y,\epsilon}(n),\vert S(t'_{y,\epsilon}(n))\vert\leq n^{1/2-\epsilon/8}).
\end{multline*}
By \eqref{bound_stopping_time} and \eqref{small_y'},
\begin{align*}
\pr(t'_{y,\epsilon}(n)\leq n^{1-\epsilon}, \tau'_{y}>t'_{y,\epsilon}(n),&\vert S(t'_{y,\epsilon}(n))\vert\leq n^{1/2-\epsilon/8})\\
&\geq \pr(t'_{y,\epsilon}(n)\leq n^{1-\epsilon}, \tau'_{y}>n^{1-\epsilon})-Cn^{-(r-2)/2}\\
&\geq \pr(\tau'_{y}>n^{1-\epsilon})-(\exp(-Cn^{\epsilon})-n^{-(r-2)/2}).
\end{align*}
Since $r>p+2$, applying Lemma \ref{lowerBoundProbSurv} yields that
\begin{equation*}
\pr(t'_{y,\epsilon}(n)\leq n^{1-\epsilon}, \tau'_{y}>t'_{y,\epsilon}(n), \vert S(t'_{y,\epsilon}(n))\vert\leq n^{1/2-\epsilon/8})
\geq c n^{-p/2(1-\epsilon)}.
\end{equation*} 
Hence,
\begin{equation}\label{lower_bound_local_proba}
\pr(y+S'(n)=x,t'_{y,\epsilon}(n)\leq n^{1-\epsilon},\tau'_{y}>n,\vert S(t'_{y,\epsilon}(n))\vert\leq n^{1/2-\epsilon/8})\geq cn^{-p-d/2+O(\epsilon)}
\end{equation}
for some constant $c$ independent of $y\in{K}$ with $\vert y\vert \leq A\sqrt{n}$. Therefore, since $r>p+2$, choosing $\epsilon>0$ small enough yields
\begin{equation*}
     O(\epsilon)<\frac{r-2}{2},
\end{equation*} 
and thus 
\begin{multline*}
     \pr(y+S'(n)=x,t'_{y,\epsilon}(n)\leq n^{1-\epsilon},\vert S(t'_{y,\epsilon}(n))\vert> n^{1/2-\epsilon/8},\tau'_{y}>n)\\
=o\bigl(\pr(y+S'(n)=x,t'_{y,\epsilon}(n)\leq n^{1-\epsilon},\tau'_{y}>n)\bigr).
\end{multline*}
Hence, finally
\begin{multline*}\pr(x+S(n)=y,\tau_{x}>n)
\sim\\\kappa V(x)n^{-p-d/2}\e[u(y'_{\epsilon}(n))\exp(-\vert y\vert^{2}/(2n));t'_{y,\epsilon}(n)\leq n^{1-\epsilon}, \tau'_{y}>t'_{y,\epsilon}(n),
\vert S(t'_{y,\epsilon}(n))\vert\leq n^{1/2-\epsilon/8}].
\end{multline*}
We have again 
\begin{align*}
\e[u(y'_{\epsilon}(n))\exp(-\vert& y'_{\epsilon}(n)\vert^{2}/(2n)),t'_{y,\epsilon}(n)> n^{1-\epsilon}, \tau'_{y}>t'_{y,\epsilon}(n)]\\
&\leq n^{p/2}\sup_{y\in{K}}\left\vert u(y)\exp(-\vert y\vert^{2}/2)\right\vert \pr(t'_{y,\epsilon}(n)> n^{1-\epsilon}, \tau'_{y}>t'_{y,\epsilon}(n))\\
&\leq Cn^{p/2}\exp(-Cn^{\epsilon}),
\end{align*}
for some constant $C>0$, and by \eqref{small_y'}
\begin{align*}
\e[u(y'_{\epsilon}(n))&\exp(-\vert y'_{\epsilon}(n)\vert^{2}/(2n));t'_{y,\epsilon}(n)\leq n^{1-\epsilon}, \tau'_{y}>t'_{y,\epsilon}(n),
\vert S(t'_{y,\epsilon}(n))\vert\geq n^{1/2-\epsilon/8}]\\
&\leq n^{p/2}\limsup_{\substack{y\in{K}\\\dist(y,\sigma)=o(1)}}\left\vert u(y)\exp(-\vert y\vert^{2}/2)\right\vert \pr(t'_{y,\epsilon}(n)\leq n^{1-\epsilon}, \tau'_{y}>t_{y,\epsilon}(n),\vert S(t'_{y,\epsilon}(n))\vert\geq n^{1/2-\epsilon/8})\\
&\leq Cn^{(p+2-r)/2}
\end{align*}
for some constant $C>0$, with $r>p+2$. Hence, by \eqref{lower_bound_local_proba},
\begin{equation*}
     \pr(x+S(n)=y,\tau_{x}>n)\sim\kappa V(x)n^{-p-d/2}\e[u(y'_{\epsilon}(n))\exp(-\vert y\vert^{2}/(2n)); \tau'_{y}>t'_{y,\epsilon}(n)].\qedhere
\end{equation*}
\end{proof}

\appendix
\section{Regularity and estimates for the heat kernel in a cone}
\label{appendix_heat_estimates}

In this appendix, we prove several inequalities concerning the heat kernel in a cone. We start with the boundedness of the gradient of $u$.
\begin{lemma}\label{bound_gradient_u}
There exists $C>0$ such that 
\begin{equation*}\vert \nabla u(z)\vert\leq C\vert z\vert^{p-1}.\end{equation*}
\end{lemma}
\begin{proof}
If $\mathbf{S}$ is a sphere and $s\in \mathbf{S}$, denote by $n(s)$ the outward normal vector at $s$. Since ${K}$ is convex, by \cite[Eq.~(0.2.3)]{Va-99} there exists a constant $C'$ such that for all $z\in {K}$,
\begin{equation}\label{upper_bound_u}
u(z)\leq C'\dist(z,\partial{K})\vert z\vert^{p-1}.
\end{equation} 
Set $r=\dist(z,\partial {K})$. If $v$ is a unit vector in ${\bf R}^{d}$, $\partial_{v}u$ is harmonic and we have
\begin{align*}
\vert \partial_{v}u(z)\vert&=\frac{1}{\Vol(B(z,r))}\left\vert \int_{B(z,r)}\partial_{v}u(s)ds\right\vert\\
&=\frac{1}{\Vol(B(z,r))}\left\vert\int_{\partial B(z,r)}u(s)\langle v,n(s)\rangle ds\right\vert \\
&\leq \frac{\Vol(\partial B(z,r))}{\Vol(B(z,r))}C\vert z\vert^{p-1}r,
\end{align*}
for some positive constant $C$, where we have used \eqref{upper_bound_u} at the last line. Hence, since $\frac{\Vol(\partial B(z,r))}{\Vol(B(z,r))}=C\frac{1}{r}$ for some $C$, we have
\begin{equation*}\vert\partial_{v}u(z)\vert \leq C\vert z\vert^{p-1}\end{equation*}
for some constant $C>0$.
\end{proof}
The latter result yields in the next lemma the H\"older continuity of $u$. In the following statement, $u$ is said locally Lipschitz with local Lipschitz constant $A(z)$ at $z\in{K}$ if $u_{\vert B(z,1)}$ is Lipschitz with Lipschitz constant $A(z)$.
\begin{lemma}\label{Holder_u}
The function $u$ is locally Lipschitz at $z\in{K}$ with local Lipschitz constant $(1+\vert z\vert)^{p-1}$.
\end{lemma}

\begin{proof}
Let $y,y'\in{K}$ with $\vert y-y'\vert \leq 2$. Then, by Lemma \ref{bound_gradient_u}, 
\begin{equation*}\left\vert u(y)-u(y')\right\vert=\left \vert \int_{0}^{1}\langle\nabla u(y'+t(y-y')),y-y'\rangle dt\right\vert  \leq \vert y'-y\vert C\int_{0}^{1}\vert y'+t(y-y')\vert^{p-1}dt.\end{equation*}
Since $K$ is convex, $p\geq 1$ and there exists a constant $C'$ such that $\int_{0}^{1}\vert y'+t(y-y')\vert^{p-1}dt\leq C'\max\{\vert y'\vert,\vert y\vert\}^{p-1}$, so that 
\begin{equation*}\vert u(y)-u(y')\vert \leq C\max\{\vert y'\vert,\vert y\vert\}^{p-1}\vert y-y'\vert\end{equation*}
for some positive constant $C$. For $y',y\in B(z,1)$, we have $\max\{\vert y'\vert,\vert y\vert\}\leq 1+\vert z\vert$, and thus
\begin{equation*}\vert u(y)-u(y')\vert \leq C(1+\vert z\vert)^{p-1}\vert y-y'\vert\end{equation*}
for some constant $C>0$ independent of $y,y'\in B(z,1)$.
\end{proof}

The following lemma gives the proof of \eqref{lower_bound_ball_C}.
\begin{lemma}\label{proof_lower_bound_ball_C}
There exists $c>0$ such that for all $t>0$,
\begin{equation*}
     \inf_{z\in{K}}V(z,\sqrt{t})>ct^{d/2}.
\end{equation*}
\end{lemma}
\begin{proof}
Write ${K}$ as the epigraph of a Lipschitz function $\phi:H\to{\bf R}$ with Lipschitz constant $L$, where $H$ is a hyperplane of ${\bf R}^{d}$. Let $v$ be the unit vector normal to $H$ pointing toward the epigraph of $\phi$; then one has $x+z\in{K}$ for all $x\in{K}$ and $z\in{\bf R}^{d}$ such that $\langle z,v\rangle \geq \sqrt{1-L^{-2}}\vert z\vert$. Therefore, for some positive constant $c$, one has
\begin{equation*}
     V(x,\sqrt{t})\geq \Vol \lbrace \vert z\vert \leq \sqrt{t},\langle z,v\rangle \geq \sqrt{1-L^{-2}}\vert z\vert\rbrace>ct^{d/2}. 
\end{equation*}
Note that if ${K}$ is convex, then $z+\left(B(0,\sqrt{t})\cap {K}\right)\subset{K}$ for all $z\in {K}$, so that we can simply choose $c=V(0,1)$. 
\end{proof}
From the Gaussian estimates, we can deduce the H\"{o}lder regularity of $\bar{k}$ and $\bar{K}$. In order to obtain this regularity, we first need to give a uniform bound for the ratio of $k_{t}$ and $k_{t'}$ for $t'\leq t$. This is done in the following lemma, using the fundamental volume doubling property of the r\'eduite $u$ (see \cite[Thm~4.19]{GySal-11}): there exists a constant $D>0$ such that 
\begin{equation}\label{doubling_volume_property}
\int_{B(x,2r)\cap {K}}u(z)^{2}dz\leq D\int_{B(x,r)\cap{K}}u(z)^{2}dz
\end{equation}
for all $x\in{K}$ and $r>0$.
\begin{lemma}\label{boundScalingk}
For all $t\in(0,1)$, there exists $c_{t}>0$ such that for all $x\in{K}$ and $s\in [t,1]$,
\begin{equation*}
     c_{t}\leq \frac{\bar{k}(sx)}{\bar{k}(x)}\leq 1.
\end{equation*}
\end{lemma}
\begin{proof}
Since $\bar{k}(sx)=k_{s^{-2}}(x)=\pr(\tau^{\bm}_{x}> s^{-2})$, $\bar{k}(sx)$ is increasing with $s$ and we only have to prove the bound for $s=t$. Let $n\geq 1$ be such that $0<2^{-n}\leq t$.
Let us prove the result for $r:=2^{-n}$, which implies the result for $t$. By \eqref{estimateWithVolume}, we have
\begin{equation*}c' \sqrt{\frac{V(rx,1)}{V(x,1)}}\sqrt{\frac{\int_{B(x,1)\cap {K}}u(z)^{2}dz}{\int_{B(rx,1)\cap {K}}u(z)^{2}dz}}\frac{u(rx)}{u(x)}\leq \frac{\bar{k}(rx)}{\bar{k}(x)} \leq 1\end{equation*}
for $c'=\frac{c'_{1}}{C'_{1}}$. 
By \eqref{lower_bound_ball_C} and the fact that $V(z,1)\leq \Vol(B(0,1))$ for all $z\in{K}$,  
\begin{equation*} \sqrt{\frac{V(rx,1)}{V(x,1)}}\geq \sqrt{\frac{c}{\Vol(B(0,1))}}.\end{equation*}
Moreover, by the scaling property of $u$,
\begin{equation*}\int_{B(rx,1)\cap {K}}u(z)^{2}dz=r^{d}\int_{B(x,r^{-1})\cap {K}}u(rz)^{2}dz=r^{d+2p}\int_{B(x,r^{-1})\cap {K}}u(z)^{2}dz,\end{equation*}
where we have used that $r^{-1}\left(B(rx,1)\cap {K}\right)=B(x,r^{-1})\cap {K}$ by the cone property. Using $n$ times the doubling volume property \eqref{doubling_volume_property} yields 
\begin{equation*}\int_{B(x,r^{-1})\cap {K}}u(z)^{2}dz=\int_{B(x,2^{n})\cap {K}}u(z)^{2}dz\leq D^{n}\int_{B(x,1)\cap {K}}u(z)^{2}dz,\end{equation*}
independently of $x\in {K}$. Therefore,
\begin{equation*}\int_{B(rx,1)\cap {K}}u(z)^{2}dz\leq r^{d+2p}D^{n}\int_{B(x,1)\cap {K}}u(z)^{2}dz,\end{equation*} 
and thus 
\begin{equation*}\sqrt{\frac{\int_{B(x,1)\cap {K}}u(z)^{2}dz}{\int_{B(rx,1)\cap {K}}u(z)^{2}dz}}\geq 2^{nd/2+np}D^{-n/2}.\end{equation*}
By the scaling property of $u$, $\frac{u(rx)}{u(x)}=r^p=2^{-np}$, which finally give
\begin{equation*}
     c'2^{nd/2}D^{-n/2}\sqrt{\frac{c}{\Vol(B(0,1))}}\leq \frac{\bar{k}(rx)}{\bar{k}(x)}\leq 1.\qedhere
\end{equation*}
\end{proof}
We say that a function $f$ is locally H\"{o}lder at $z\in{K}$ with exponent $\alpha$ and constant $C(z)$ if $\vert f(z')-f(z)\vert\leq C(z)\vert z'-z\vert^{\alpha}$ for $z'$ such that $\vert z-z'\vert\leq 1$. The function is said globally H\"{o}lder on $D\subset {K}$ with exponent $\alpha$ and constant $C$ if $\vert f(z')-f(z)\vert\leq C\vert z'-z\vert^{\alpha}$ for $z,z'\in{D}$.
\begin{proposition}\label{holderContinuity}
There exist $0< \alpha\leq 1$ and $C_{\alpha}>0$ such that $k$ and $\frac{1}{\bar{k}(x)}\bar{K}(x,\cdot)$  are locally H\"{o}lder at each $z$ in ${K}$ with exponent $\alpha$ and respective constants 
\begin{equation*}
     C_{\alpha}(1+\vert z\vert^{p-1})\quad \text{and}\quad C_{\alpha}(1+\vert z\vert^{p-1})\exp(-\vert z-x\vert^{2}/(2c_{3})).
\end{equation*}
By symmetry, $\bar{K}(\cdot,y)$ is also locally $\alpha$-H\"{o}lder at $z\in{K}$ with H\"{o}lder constant 
\begin{equation*}
     C_{\alpha}(1+\vert z\vert^{p-1})\exp(-\vert z-y\vert^{2}/(2c_{3})).
\end{equation*} 
\end{proposition}
\begin{proof}
Let us start by the H\"{o}lder continuity of $\frac{1}{\bar{k}(x)}\bar{K}(x,\cdot)$. Let $x,y,y'\in{K}$ be such that $\vert y-y'\vert\leq 1$ and $\vert y'\vert \geq \vert y\vert$. By the second part of \cref{estimateK}, we have
\begin{equation*}\left\vert \frac{\bar{K}(x,y)}{u(y)}-\frac{\bar{K}(x,y')}{u(y')}\right\vert\leq C_{4}\vert y-y'\vert^{\alpha}\frac{K_{2}(x,y)}{u(y)}.\end{equation*}
Hence
\begin{equation*}\frac{1}{\bar{k}(x)}\left\vert \bar{K}(x,y)-\bar{K}(x,y')\right\vert\leq C_{4}\frac{1}{\bar{k}(x)}\vert y-y'\vert^{\alpha}K_{2}(x,y)+\frac{1}{\bar{k}(x)}\bar{K}(x,y')\left\vert\frac{u(y)}{u(y')}-1\right\vert .\end{equation*}
Using the first part of \cref{estimateK} and then \cref{estimatek} gives
\begin{equation*}\frac{1}{\bar{k}(x)}\bar{K}(x,y')\leq C_{2}\frac{\bar{k}(y')}{\sqrt{V(x,1)V(y',1)}}\exp(-\vert x-y'\vert^{2}/c_{3})\leq C\frac{u(y')}{u(y'_{1})}\exp(-\vert x-y'\vert^{2}/c_{3}),\end{equation*}
for some constant $C$, where we have also used in the last inequality that $V(z,1)\geq c$ for some positive constant $c$ independent of $z$ (see the previous proof for a proof of this fact). Likewise,
\begin{align*}
\frac{1}{\bar{k}(x)}K_{2}(x,y)&\leq C_{2}\frac{\bar{k}(x/\sqrt{2})}{\bar{k}(x)}\frac{\bar{k}(y/\sqrt{2})}{\sqrt{V(x,\sqrt{2})V(y,\sqrt{2})}}\exp(-\vert x-y\vert^{2}/(2c_{3}))\\
&\leq C\frac{\bar{k}(x/\sqrt{2})}{\bar{k}(x)}\exp(-\vert x-y\vert^{2}/(2c_{3}))\leq C\exp(-\vert x-y\vert^{2}/(2c_{3})),
\end{align*}
where we have used Lemma \ref{boundScalingk} in the last inequality. Thus
\begin{multline*}
     \frac{1}{\bar{k}(x)}\left\vert \bar{K}(x,y)-\bar{K}(x,y')\right\vert\\
     \leq C\vert y-y'\vert^{\alpha}\exp(-\vert x-y\vert^{2}/(2c_{3}))+C\frac{1}{u(y'_{1})}\exp(-\vert x-y'\vert^{2}/c_{3})\left\vert u(y)-u(y')\right\vert.
\end{multline*}
By Lemma \ref{Holder_u}, for $y,y'\in{K}$ such that $\vert y-y'\vert \leq 1$, we have for some constant $C'>0$
\begin{equation*}\vert u(y)-u(y')\vert\leq  C(1+\vert y'\vert)^{p-1}\vert y-y'\vert.\end{equation*}
Since $\dist(y'_{1},\partial {K})\geq c_{0}$ by definition, \cite[Lem.~19]{DeWa-15} yields that for some constant $c>0$ 
\begin{equation}
u(y'_{1})\geq cc_{0}^{p}.
\end{equation}
Hence, there exists $C^{(3)}$ such that 
\begin{equation*}\frac{1}{u(y'_{1})}\left\vert u(y)-u(y')\right\vert\leq C^{(3)} (1+\vert y'\vert)^{p-1}\vert y-y'\vert^{1}.\end{equation*}
Therefore, there exists $C^{(4)}>0$ such that
\begin{equation*}\frac{1}{\bar{k}(x)}\left\vert \bar{K}(x,y)-\bar{K}(x,y')\right\vert\leq C^{(4)}\vert y-y'\vert^{\alpha}(1+ \vert y'\vert^{p-1})\exp(-\vert x-y\vert^{2}/(2c_{3})),\end{equation*}
which proves the first part of the lemma. Since $\bar{k}$ is bounded by $1$, $\bar{K}(x,\cdot)$ is also locally $\alpha$-H\"{o}lder with the same constants, and by symmetry the same holds for $K(\cdot,y)$.
Therefore, for $x,x'\in{K}$ such that $\vert x-x'\vert\leq 1$,
\begin{align*}
\left\vert \bar{k}(x)-\bar{k}(x')\right\vert&\leq \int_{{K}}\vert \bar{K}(x,y)-\bar{K}(x',y)\vert dy\\
&\leq C^{(4)}(1+\vert x'\vert^{p-1})\vert x-x'\vert^{\alpha}\int_{{K}}\exp(-\vert x'-y\vert^{2}/(2c_{3}))dy\\
&\leq C_{\alpha}(1+\vert x'\vert^{p-1})\vert x-x'\vert^{\alpha}
\end{align*}
for some constant $C_{\alpha}\geq 0$.
\end{proof}
\begin{remark}\label{globallyHolder}
Since $\exp(-\vert z-y\vert^{2}/(2c_{3}))\leq 1$ for all $x,z\in{K}$, Proposition \ref{holderContinuity} implies that $\frac{1}{\Bar
k(x)}\bar{K}(x,\cdot)$ and $\bar{K}(\cdot,x)$ are globally H\"{o}lder on any bounded subset $D$ of ${K}$ with a H\"{o}lder constant $C(D)$ independent of $x\in{K}$.
\end{remark}
We end this section by a uniform (in time) estimate of the convergence of the heat kernel.
\begin{lemma}\label{equivalenceKernelTime}
Let $t_{0}>0$ and $M>0$. Uniformly in $t>t_{0}$ and $x,y\in{K}$ with $\vert y-x\vert\leq M$,
\begin{equation}
\label{eq:first_conv}
K_{t+h}(x,y)\sim_{h\rightarrow 0} K_{t}(x,y),
\end{equation}
and uniformly in $t>t_{0}$ and $x\in{K}$,
\begin{equation}
\label{eq:second_conv}
     k_{t+h}(x)\sim_{h\rightarrow 0}k_{t}(x).
\end{equation}
\end{lemma}
\begin{proof}
For the first convergence \eqref{eq:first_conv}, let $t>t_{0}$ and $x,y\in{K}$ with $\vert y-x\vert\leq M$. Then for $h\geq 0$,
\begin{equation*}\left\vert K_{t+h}(x,y)-K_{t}(x,y)\right\vert\leq \int_{0}^{h}\vert \partial_{t}K_{t+r}(x,y)\vert dr.\end{equation*}
Hence, by the last part of \cref{estimateK}, 
\begin{align*}
&\left\vert K_{t+h}(x,y)-K_{t}(x,y)\right\vert\\
&\leq C_{5}\int_{0}^{h} \frac{k_{t+r}(x)k_{t+r}(y)}{(t+r)\sqrt{V(x,\sqrt{t+r})V(y,\sqrt{t+r})}}\left(1+\frac{\vert x-y\vert^{2}}{t+r}\right)^{\beta+1}\exp\left(-\frac{\vert x-y\vert^{2}}{4(t+r)}\right)dr\\
&\leq C_{5}\int_{0}^{h} \frac{k_{t}(x)k_{t}(y)}{(t+r)\sqrt{V(x,\sqrt{t+r})V(y,\sqrt{t+r})}}\left(1+\frac{\vert x-y\vert^{2}}{t+r}\right)^{\beta+1}\exp\left(-\frac{\vert x-y\vert^{2}}{4(t+r)}\right)dr,
\end{align*}
where we have used that $k_{t}$ is decreasing in $t$. Applying now the lower bound in the first part of \cref{estimateK} and then \eqref{lower_bound_ball_C}, we get, for $\vert h\vert\leq t_{0}/2$,
\begin{align*}
\left\vert K_{t+h}(x,y)-K_{t}(x,y)\right\vert&
\leq \frac{C_{5}}{c_{2}}\int_{0}^{h}K_{t}(x,y)\frac{1}{t+r}\left(1+\frac{\vert x-y\vert^{2}}{t+r}\right)^{\beta+1}\times\\
&\quad\times\frac{\sqrt{V(x,\sqrt{t})V(y,\sqrt{t})}}{\sqrt{V(x,\sqrt{t+r})V(y,\sqrt{t+r})}}\exp\left(-\frac{\vert x-y\vert^{2}}{4(t+r)}+\frac{\vert x-y\vert^{2}}{C_{3}t}\right) dr\\&
 \leq C\int_{0}^{h}\frac{t^{d/2}}{(t+r)^{d/2}}\frac{1}{t+r}K_{t}(x,y)\left(1+\frac{2M^{2}}{t_{0}}\right)^{\beta+1}\exp\left(\frac{M^{2}}{C_3t_0}\right) dr,
\end{align*}
where $C$ is a positive constant. Hence, for some other constant $C$,
\begin{equation*}\left\vert K_{t+h}(x,y)-K_{t}(x,y)\right\vert\leq CK_{t}(x,y)\left(1-\frac{t^{d/2}}{(t+h)^{d/2}}\right),\end{equation*}
and as $h$ goes to zero such that $\vert h\vert\leq t_{0}/2$, $K_{t+h}(x,y)\sim K_{t}(x,y)$ uniformly in $t>t_{0}$, $x,y\in{K}$ with $\vert y-x\vert\leq M$. 

For the second convergence \eqref{eq:second_conv} in the statement of Lemma \ref{equivalenceKernelTime}, we first remark that by homogeneity, 
\begin{equation*}
     \frac{k_{t+h}(x)}{k_{t}(x)}=\frac{k_{1+h/t}(x/\sqrt{t})}{k_{1}(x/\sqrt{t})}.
\end{equation*}
Hence, it suffices to prove \eqref{eq:second_conv} for $t_{0}>0$, uniformly for all $x\in K$. By Lemma \ref{boundScalingk}, there exists $c'>0$ such that for all $\vert h\vert \leq t_{0}/2$ and $x\in K$,
\begin{equation}\label{using_lemma_18}
k_{t_{0}+h}(x)=\bar{k}(x/\sqrt{t_{0}+h})\leq c'\bar{k}(x/\sqrt{t_{0}})=c'k_{t_{0}}(x).
\end{equation}
Let $\delta>0$ and $M>0$ be such that 
\begin{equation}
\label{choice_M}
     C_{2}(1+c^{2})\frac{1}{c(t_{0}/2)^{d/2}}\int_{\vert z\vert>M}\exp\left(-\frac{2\vert z\vert^{2}}{3c_{3}t_{0}}\right)dz\leq \delta,
\end{equation}
where $c$ is the constant coming from \eqref{lower_bound_ball_C}. Based on \eqref{eq:first_conv}, let $0<h_{0}<t_{0}/2$ be such that for all $x,y\in{K}$ with $\vert y-x\vert \leq M$, and $h\in{\bf R}$ such that $\vert h\vert \leq h_{0}$,
\begin{equation*}(1-\delta)K_{t_{0}}(x,y)\leq K_{t_{0}+h}(x,y)\leq (1+\delta)K_{t_{0}}(x,y).\end{equation*} 
Then for such $x\in{K}$ and $0\leq \vert h\vert\leq h_{0}$, we have by the first part of the lemma 
\begin{align*}
\vert k_{t_{0}+h}(x)-k_{t_{0}}(x)\vert&\leq \int_{\vert z\vert \leq M}\vert K_{t_{0}+h}(x,x+z)-K_{t_{0}}(x,x+z)\vert dz\\
&\hspace{2cm}+\int_{\vert z\vert \geq M}\vert K_{t_{0}+h}(x,x+z)-K_{t_{0}}(x,x+z)\vert dz\\
&\leq \int_{\vert z\vert \leq M}\delta K_{t_{0}}(x,x+z) dz+\int_{\vert z\vert \geq M}\vert K_{t_{0}}(x,x+z)\vert dz\\
&\hspace{2cm}+\int_{\vert z\vert \geq M}\vert K_{t_{0}+h}(x,x+z)\vert dz\\
&\leq \delta k_{t_{0}}(x)+\int_{\vert z\vert \geq M} K_{t_{0}}(x,x+z) dz+\int_{\vert z\vert \geq M} K_{t_{0}+h}(x,x+z) dz.
\end{align*}
By Theorem \ref{estimateK}, the two last terms of the latter inequality are bounded by
\begin{multline}
\int_{\vert z\vert \geq M} K_{t_{0}}(x,x+z)dz+\int_{\vert z\vert \geq M} K_{t_{0}+h}(x,x+z) dz\\
\leq C_{2}\int_{\vert z\vert \geq M}\left\{\frac{k_{t_{0}}(x+z)k_{t_{0}}(x)}{ct_{0}^{d/2}}\exp\left(-\frac{\vert z\vert^{2}}{c_{3}t_{0}}\right)+\frac{k_{t_{0}+h}(x+z)k_{t_{0}+h}(x)}{c(t_{0}+h)^{d/2}}\exp\left(-\frac{\vert z\vert^{2}}{c_{3}(t_{0}+h)}\right)\right\}dz.\label{eqeqeq}
\end{multline}
Hence, using \eqref{using_lemma_18}, we obtain that \eqref{eqeqeq} is equal to or less than
\begin{align*}
&C_{2}\int_{\vert z\vert \geq M}\left\{\frac{k_{t_{0}}(x+z)k_{t_{0}}(x)}{ct_{0}^{d/2}}\exp\left(-\frac{\vert z\vert^{2}}{c_{3}t_{0}}\right)+\frac{c^{2}k_{t_{0}}(x+z)k_{t_{0}}(x)}{c(t_{0}+h)^{d/2}}\exp\left(-\frac{\vert z\vert^{2}}{c_{3}(t_{0}+h)}\right)\right\}dz\\
\quad &\leq C_{2}(1+c^{2}) k_{t_{0}}(x)\int_{\vert z\vert \geq M}\frac{1}{c(t_{0}/2)^{d/2}}\exp\left(-\frac{2\vert z\vert^{2}}{3c_{3}t_{0}}\right)dz\\
\quad&\leq\delta k_{t_{0}}(x),
\end{align*}
where the last inequality is due to the choice of $M>0$ in \eqref{choice_M}. Finally,
\begin{equation*}\vert k_{t_{0}+h}(x)-k_{t_{0}}(x)\vert\leq 2\delta k_{t_{0}}(x).\end{equation*}
Now, if $t>t_{0}$ and $\vert h\vert \leq h_{0}$, then by the previous result
\begin{equation*}
     \vert k_{t+h}(x)-k_{t}(x)\vert=\vert k_{t_{0}+h(t_{0}/t)}(x/\sqrt{t/t_{0}})-k_{t_{0}}(x/\sqrt{t/t_{0}})\vert
\leq 2\delta k_{t_{0}}(x/\sqrt{t/t_{0}})\leq 2\delta k_{t}(x),
\end{equation*}
where we have used the fact that $\vert h(t_{0}/t)\vert\leq h_{0}$ for $h\leq h_{0}$ and $t\geq t_{0}$. Hence, the uniform convergence holds for all $x\in K$ and $t\geq t_{0}$.
\end{proof}
\section{Fuk-Nagaev inequalities in a cone}\label{Fuk-Nagaev-cone}
We give here a useful generalization of \cite[Lem.~24]{DeWa-15} under stronger moment conditions on the increments. This lemma will also be used in a future work (see \cite{DuRaTaWa-20}).

\begin{lemma}\label{Lemma_24_revisited}
Let $0\leq \alpha\leq p$ and $A>0$, and suppose that the increment $X$ admits moments of order $q>\alpha$. Set 
\begin{equation*}S(x,n)^{+}=\sup_{1\leq \ell\leq n^{1-\epsilon}}\{\vert S(\ell)\vert: \tau_{x}>\ell\}.\end{equation*} 
Then, for each $s<(q-\alpha)/2$ and $\beta\in ((p/2-1)\vee 0,p/2)$, there exists $C>0$ such that
\begin{equation*}
     \e\bigl[(S(x,n)^{+})^{\alpha};S(x,n)^{+}\geq n^{1/2-\epsilon/8}\bigr]\leq Cn^{-s}n^{(1-(p/2-\beta))}(1+\vert x\vert)^{p-2\beta}
\end{equation*}
for all $x\in K$. In particular, uniformly on $x\in K$ with $\vert x\vert\leq A\sqrt{n}$,
\begin{equation*}
     \e\bigl[(S(x,n)^{+})^{\alpha};S(x,n)^{+}\geq n^{1/2-\epsilon/8}\bigr]\leq Cn^{-s+1}.
\end{equation*}
\end{lemma} 
The proof of the lemma follows closely the one of \cite[Cor.~23]{DeWa-15}: the main ingredient are Fuk-Nagaev inequalities in the form of \cite[Lem.~24]{DeWa-15} (see also \cite{FuNa-71}), which we now recall:
\begin{equation}\label{Fuk-Nagaev}
\pr\left(\vert S(n)\vert>x, \max_{1\leq k\leq n}\vert X(k)\vert\leq y\right)\leq 2d\exp\left(\frac{x}{\sqrt{d}y}\right)\left(\frac{\sqrt{d}n}{xy}\right)^{x/(\sqrt{d}y)}.
\end{equation}
Moreover, we will use the following bounds from  \cite[Thm~3.1]{Mc-84} on the moments of the exit time: for all $r<p$, there exists $C>0$ such that
\begin{equation}\label{exit_time}
\e[\tau_{x}^{r/2}]<C(1+\vert x\vert)^{r}.
\end{equation}
\begin{proof}
Let $\ell_{n}\leq n^{1-\epsilon}$ be the time where $S(x,n)^{+}$ is reached by the random walk and 
\begin{equation*}
     \mu_{n}=\inf\{j\geq 1: \vert X(j)\vert\geq n^{1/2-\epsilon/4}\}.
\end{equation*}
We have
\begin{align*}
     \e\bigl[(S(x,n)^{+})^{\alpha};S(x,n)^{+}\geq n^{1/2-\epsilon/8}\bigr]&= \e\bigl[(S(x,n)^{+})^{\alpha};S(x,n)^{+}\geq n^{1/2-\epsilon/8},\mu_{n}>\ell_{n} \bigr]\\
&\quad+\sum_{j=1}^{n^{1-\epsilon}}\e\bigl[(S(x,n)^{+})^{\alpha};S(x,n)^{+}\geq n^{1/2-\epsilon/8},j\leq \ell_{n},\mu_{n}=j\bigr]\\
&= E_{1}+E_{2}.
\end{align*}
Remark first that 
\begin{equation*}
     E_{1}\leq \sum_{\ell=1}^{n}\e\bigl[\vert S(\ell)\vert^{\alpha},\vert S(\ell)\vert> \ell^{1/2-\epsilon/8},\max_{1\leq k\leq \ell}\vert X(k)\vert\leq n^{1/2-\epsilon/4}\bigr].
\end{equation*}
Hence, by Hölder's inequality and \eqref{Fuk-Nagaev}, the term $E_{1}$ is bounded by $\exp(-Cn^{\epsilon'})$ for some positive constants $C$ and $\epsilon'$ independent of $x\in K$. Then, we bound the second term as 
\begin{multline*}
E_{2}\leq C\sum_{j=1}^{n^{1-\epsilon}}\e\bigl[\vert S(\ell_{n})-S(j)\vert^{\alpha} ;\tau_{x}\geq j,\ell_{n}\geq j, \mu_{n}=j]\\+\e\bigl[\vert X(j)\vert^{\alpha} ;\tau_{x}\geq j-1, \mu_{n}=j\bigr]
+C\e\bigl[\vert S(j-1)\vert^{\alpha} ;\tau_{x}\geq j, \mu_{n}=j\bigr].
\end{multline*}
Using Doob maximal inequality (together with Hölder's inequality for $\alpha<2$) and then Markov inequality, we get
\begin{align*}
\sum_{j=1}^{n^{1-\epsilon}}\e\bigl[\vert S(\ell_{n})-S(j)\vert^{\alpha} ;\tau_{x}&\geq j,\ell_{n}\geq j,\mu_{n}=j\bigr]\\
&\quad\leq Cn^{\alpha(1-\epsilon)/2}\sum_{j=1}^{n^{1-\epsilon}}\pr(\tau_{x}\geq j)\pr(\vert X(j)\vert\geq n^{1/2-\epsilon/4})\\
&\quad\leq  Cn^{\alpha(1-\epsilon)/2}n^{-q(1/2-\epsilon/4)}\sum_{j=1}^{n^{1-\epsilon}}\pr(\tau_{x}\geq j).
\end{align*}
Then, by Markov inequality,
\begin{align*}
\sum_{j=1}^{n^{1-\epsilon}}\e\bigl[\vert X(j)\vert^{\alpha} ;\tau_{x}\geq j-1,\mu_{n}=j\bigr]&\leq \e(\vert X\vert^{\alpha},\vert X\vert \geq n^{1/2-\epsilon/4})\sum_{j=1}^{n^{1-\epsilon}}\pr(\tau_{x}\geq j-1)\\
&\leq Cn^{-(q-\alpha)(1/2-\epsilon/4)}\sum_{j=1}^{n^{1-\epsilon}}\pr(\tau_{x}\geq j-1).
\end{align*}
Using again \eqref{Fuk-Nagaev} yields
\begin{align*}
\e\bigl[\vert S(j-1)\vert^{\alpha} ;\tau_{x}\geq j, \mu_{n}=j\bigr]&=\e\bigl[\vert S(j-1)\vert^{\alpha};\vert S(j-1)\vert>n^{1/2-\epsilon/8} ,\tau_{x}\geq j, \mu_{n}=j\bigr]\\
&+\e\bigl[\vert S(j-1)\vert^{\alpha};\vert S(j-1)\vert\leq n^{1/2-\epsilon/8} ,\tau_{x}\geq j, \mu_{n}=j\bigr]\\
&\leq C\exp(-n^{\epsilon'})+Cn^{\alpha(1/2-\epsilon/8)}\pr(\tau_{x}\geq j-1){\bf P}(\mu_{n}=j)\\
&\leq Cn^{\alpha(1/2-\epsilon/8)}n^{-q(1/2-\epsilon/4)}\pr(\tau_{x}\geq j-1),
\end{align*}
and thus 
\begin{equation*}\sum_{j=1}^{n^{1-\epsilon}}\e\bigl[\vert S(j-1)\vert^{\alpha} ;\tau_{x}\geq j, \mu_{n}=j\bigr]\leq \sum_{j=1}^{n^{1-\epsilon}}Cn^{-(q-\alpha)/2+(2q-\alpha)\epsilon/8}\pr(\tau_{x}\geq j-1).\end{equation*}
Hence, we have
\begin{equation*}E_{2}\leq Cn^{-(q-\alpha)/2+g(\epsilon)}\sum_{j=1}^{n^{1-\epsilon}}\pr(\tau_{x}\geq j-1),\end{equation*}
with $g:{\bf R}\rightarrow {\bf R}$ linear. Then, we have  for $\beta\in ((p/2-1)\vee 0,p/2)$,
\begin{align*}
E_{2}\leq&\ Cn^{-(q-\alpha)/2+g(\epsilon)+(1-(p/2-\beta))}\sum_{j=1}^{n^{1-\epsilon}}j^{-1+(p/2-\beta)}\pr(\tau_{x}\geq j-1)\\
\leq&\ Cn^{-(q-\alpha)/2+g(\epsilon)+(1-(p/2-\beta))}\e(\tau_{x}^{p/2-\beta})\leq Cn^{-(q-\alpha)/2+g(\epsilon)+(1-(p/2-\beta))}(1+\vert x\vert)^{p-2\beta},
\end{align*}
where we have used \eqref{exit_time} in the last inequality. This yields the first part of the lemma for $\epsilon$ small enough. Moreover, for any $s<(q-\alpha-2)/2$ we have by choosing $\epsilon$ small enough
\begin{equation*}E_{2}\leq Cn^{-s+1},\end{equation*}
for all $x\in K$, $\vert x\vert \leq A\sqrt{n}$.
\end{proof}

\end{document}